\documentclass[10pt]{article}
\usepackage[english]{babel}
\usepackage{amsmath}
\allowdisplaybreaks[4]
\usepackage{amsthm}
\usepackage{amsfonts}
\usepackage{makecell}
\usepackage{fancybox}
\usepackage{epsfig}
\usepackage{epsf}
\usepackage{lineno}
\usepackage{amssymb}
\usepackage{epic,eepic}
\usepackage{latexsym,bm}
\usepackage{graphicx}
\usepackage{multirow}
\usepackage{subfigure}
\graphicspath{{fig/}} 
\usepackage{float}
\usepackage{color}
\usepackage{hyperref}
\usepackage{cleveref}
\usepackage{mathrsfs}
\usepackage{amssymb} 
\usepackage{booktabs}
\usepackage{hhline}
\usepackage{adjustbox}
\usepackage{bbm}      
\usepackage[top=1.2in, bottom=1.2in, left=1in, right=1in, dvips, letterpaper]{geometry}
\usepackage{algorithm}
\usepackage{algorithmic}
\usepackage{cite}
\usepackage{boxedminipage}

\newcommand\argmin{\mathop{\textrm{argmin}}}
\newcommand\crule[1][5cm]{%
  \par
  \nointerlineskip
  \centerline{\hbox to #1{\hrulefill}}%
  \nointerlineskip}

\numberwithin{equation}{section}
\numberwithin{algorithm}{section}
\newtheorem{theorem}{{\sc Theorem}}[section]
\newtheorem{lemma}{{\sc Lemma}}[section]
\newtheorem{corollary}[theorem]{Corollary}
\newtheorem{remark}{Remark}[section]
\newtheorem{assumption}{Assumption}[section]
\newtheorem{proposition}{Proposition}[section]

\newcommand{\R}{\mathbb{R}}
\newcommand{\E}{\mathbb{E}}

\newcommand{\calL}{\mathcal{L}}

\newcommand{\be}{\begin{equation}}
\newcommand{\ee}{\end{equation}}
\newcommand{\bee}{\begin{equation*}}
\newcommand{\eee}{\end{equation*}}

\ifpdf
  \DeclareGraphicsExtensions{.eps,.pdf,.png,.jpg}
\else
  \DeclareGraphicsExtensions{.eps}
\fi

\title{A Fletcher's Augmented Lagrangian-Based Stochastic First-Order Method for Nonconvex Equality-Constrained Optimization}

\author{Yawen Cui\thanks{Dalian University of Technology, Dalian, China. ({ywcui@mail.dlut.edu.cn})}
    \and Qiankun Shi\thanks{Sun Yat-sen University, Guangzhou, China. ({shiqk@mail2.sysu.edu.cn})}
	\and Xiao Wang\thanks{Corresponding author. Sun Yat-sen University, Guangzhou, China. ({wangx936@mail.sysu.edu.cn})}  
	\and Xiantao Xiao\thanks{Dalian University of Technology, Dalian, China. ({xtxiao@dlut.edu.cn})}}

\begin{document}
%\linenumbers
\maketitle

\begin{abstract}
In this paper, we study nonconvex equality-constrained optimization problems in which only stochastic first-order approximations of the objective and constraint functions are available. Owing to the stochasticity in both  objective and constraints, most existing stochastic first-order methods incur relatively high oracle complexity, particularly in terms of stochastic constraint function evaluations. To address this issue, we develop a stochastic first-order method based on a decomposed stochastic search direction, and employ Fletcher's augmented Lagrangian as a smooth merit function for step-size selection. To cope with the possible loss of uniform nondegeneracy of the stochastic Jacobian, we introduce an event decomposition based on the smallest singular value, which enables us to control perturbations in the stochastic search direction. Under an additional Lipschitz continuity assumption on the second-order derivatives of the objective and constraint functions, we show that the proposed algorithm attains a stochastic \(\epsilon\)-KKT point with an expected total oracle complexity of 
\(\mathcal O(\epsilon^{-3})\) in terms of stochastic gradient and stochastic constraint function evaluations. Finally, we present numerical experiments to demonstrate the performance of the proposed method.
\end{abstract}


\section{Introduction}
In this paper, we consider the nonconvex  constrained  optimization problem
\begin{align}{\label{p}}
	\min_{x\in \R^n}\quad &f(x):=\mathbb{E}_{\xi}[F(x;\xi)]\notag\\
	\mbox{ s.t. }\quad &c(x):=\mathbb{E}_{\xi}[C(x;\xi)]=\mathbf{0}, 
\end{align}
where $\xi$ is a random variable in the probability space $\Xi$ and independent of $x$, and $\E_\xi $ refers to the expectation taken  {in}  $\xi$. Here $ F: \R^n\times \Xi\to \R$  and $ C:\R^n\times \Xi\to \R^m$ are continuously differentiable  with respect to $x$ but possibly nonconvex. Throughout the paper, we assume that the feasible set of (\ref{p}) is nonempty. Problems of the form (\ref{p}), commonly referred to as fully-stochastic constrained optimization, capture a broad class of real-world  applications ranging from  stochastic optimal control \cite{BJ2010}, and chance-constrained programming \cite{MW1965,PA1970} to fair machine learning \cite{OC2020,ZVGG2019} and physics–informed neural networks \cite{KGZKM2021}. When solving these models in practice, exact evaluations of the objective and constraint functions are typically unavailable, {while only stochastic approximations can be accessed}. The challenges posed by possibly nonconvex constraints and {the} stochastic nature of the problem significantly complicate both algorithm design and complexity analysis.

A related setting that has been extensively studied is the semi-stochastic case, where {the exact information of constraints can always be available.} Existing approaches include penalty and augmented Lagrangian methods \cite{CSWX2025,JW2022,LMX2026,SW2025,SWW2025,WMY2017,ZWW2025}, proximal-point methods \cite{BDL2023,BDL2025,MLY2020}, and stochastic sequential quadratic programming (SQP) methods \cite{BCRZ2021,COR2024,CRZ2024,FNMK2024,NAK2023,O2024}. The best-known oracle complexity in the semi-stochastic setting is of order $\mathcal{O}(\epsilon^{-3})$ for finding a stochastic $\epsilon$-KKT point, typically under a mean-squared smoothness assumption and suitable constraint qualification conditions \cite{CSWX2025,IAR2025,LMX2026,SW2025,SWW2025}.

Compared with the semi-stochastic setting, the {study} on fully-stochastic problems has {been} developed more recently, {but} with generally higher complexity bounds. Among existing approaches, the inexact proximal point method proposed in \cite{BDL2023} solves the nonconvex constrained problem through a sequence of stochastic convex subproblems, resulting in a double-loop framework. Given the strong feasibility assumption, this method requires an overall oracle complexity of order $\mathcal{O}(\epsilon^{-6})$ for obtaining an approximate KKT point.  A stochastic SQP method \cite{SZZ2025} has also been developed for (\ref{p}), yielding a complexity bound of the same order $\mathcal{O}(\epsilon^{-6})$ provided that a strong linear independence constraint qualification (LICQ) holds. Penalty and ALMs form another line of work, with several recent results achieving improved complexity bounds of order $\mathcal{O}(\epsilon^{-5})$. Specifically, Li et al. \cite{LCLLX2024} propose stochastic inexact augmented Lagrangian methods, in which the inner subproblems are solved by a momentum-based variance-reduced proximal stochastic gradient method followed by a postprocessing step. Different from this, Alacaoglu and Wright \cite{AW2024} and Cui et al. \cite{CWX2025} develop momentum-based methods with a single-loop structure. Under the nonsingularity condition, the former attains an $\tilde{\mathcal{O}}(\epsilon^{-5})$ sample complexity for finding a stochastic $\epsilon$-KKT point {with a varying penalty parameter setting}, while the latter adopts a two-phase strategy and achieves an $\mathcal{O}(\epsilon^{-5})$ {oracle} complexity for a stochastic $\epsilon$-KKT point {with constant parameter settings}.  %improving upon its $\mathcal{O}(\epsilon^{-6})$ bound for a stochastic $\epsilon$-stationarity point. 
%{Furthermore, the logarithmic factors in the former bound can also be removed through appropriate parameter choices based on the maximum number of iterations, starting from a nearly  feasible point.}More recently, exact penalty methods have attracted increasing attention in fully-stochastic constrained optimization. 
Cui et al. \cite{CSWX2025} introduce an adaptive penalty method that updates the penalty parameter according to the progress of the algorithm and approximately solves each penalty subproblem by a truncated stochastic prox-linear method. Under the strong LICQ condition and sample-wise smoothness assumptions, they derive oracle complexity bounds of order $\mathcal{O}(\epsilon^{-3})$ for stochastic gradient evaluations of both the objective and constraints, and $\mathcal{O}(\epsilon^{-5})$ for stochastic constraint function value evaluations. In a broader constrained setting, Shi and Wang \cite{SW2025} consider problems with both equality and inequality constraints and propose an adaptive directional decomposition method. They establish high-probability oracle complexity for finding an $\epsilon$-KKT point, requiring $\tilde{\mathcal{O}}(\epsilon^{-3})$ stochastic gradient evaluations and $\tilde{\mathcal{O}}(\epsilon^{-5})$ stochastic constraint function value evaluations under a strong MFCQ condition {and the} sample-wise smoothness assumption.

Stochastic constrained optimization problems with nonsmooth or weakly convex structures form a special class of fully-stochastic nonconvex problems \cite{LX2025,YLHLY2025,ZLXZ2026}. A notable recent work by Liu and Xu \cite{LX2025} develops an exact penalty model solved by a single-loop SPIDER-type stochastic subgradient method. Their method guarantees an $(\epsilon,\epsilon)$-KKT point with sample complexities of order $\mathcal{O}(\epsilon^{-4})$ for stochastic subgradient evaluations and $\mathcal{O}(\epsilon^{-6})$ for stochastic constraint function value evaluations, assuming a uniform Slater-type constraint qualification. %More recently, Zhang et al. \cite{ZLXZ2026} present a stochastic approximation algorithm that combines a proximal method of multipliers with quadratic approximations of the original stochastic problem. They show that, after $T$ iterations, the algorithm achieves an expected convergence rate of $\mathcal{O}(T^{-1/4})$ with respect to the metrics associated with the $\epsilon$-KKT conditions. 
Overall, existing results suggest that the oracle complexity of fully-stochastic constrained optimization remains relatively high, especially for stochastic constraint function value evaluations.

\emph{Motivation.} Reducing the oracle complexity {regarding} constraint value evaluations within a tractable stochastic first-order framework is a central challenge. {Existing methods that employ smooth merit/penalty functions, such as quadratic penalty function or augmented Lagrangian function, typically require $\epsilon^{-1}$-dependent penalty parameters, thus own higher oracle complexity bounds \cite{AW2024,CWX2025,LCLLX2024}.} Alternatively, {methods based on nonsmooth merit/penalty functions,} such as  $f(x) + \rho \|c(x)\|$, admit exactness once $\rho$ exceeds a finite threshold {under certain CQ conditions}. However, the nonsmooth term makes stochastic constraint approximation more demanding, {leading to} relatively high sample complexity for constraint evaluations \cite{CSWX2025,LX2025,SZZ2025,SW2025}. {Therefore, to better control the total oracle complexity bounds, we are motivated to search} for a merit function that unifies smoothness and exactness. Fletcher’s augmented Lagrangian \cite{Fletcher1970} serves this purpose by replacing the {seperately} updated dual variable with a continuous function of the primal variable, thereby providing a suitable {choice to help design a stochastic approximation method for solving (\ref{p}).}

\subsection{Contributions} 
This paper develops a single-loop stochastic first-order method based on Fletcher's augmented Lagrangian for fully-stochastic nonconvex equality-constrained optimization. At each iteration, the algorithm employs truncated SPIDER-type recursive estimators to generate stochastic approximations, from which a decomposed direction is computed to balance objective decrease and feasibility improvement, with the Fletcher’s augmented Lagrangian serving as a smooth merit function and {to determine} the step size. To deal with the possible degeneracy of stochastic constraint Jacobians and to control the perturbation in the stochastic search direction, we separate the event that  whether or not the sampled Jacobian satisfies a prescribed singular-value bound. We prove that the proposed method reaches a stochastic $\epsilon$-KKT point within $\mathcal{O}(\epsilon^{-2})$ iterations and establish the corresponding expected oracle complexity of $\mathcal{O}(\epsilon^{-3})$ {to evaluate} stochastic gradient and constraint function values under second-order smoothness of objective and constraint functions (Assumption \ref{a3}) and the strong LICQ condition (Assumption \ref{a4}). This complexity bound improves the state-of-the-art methods for fully-stochastic {problems}, particularly regarding the {computation of stochastic constraint function values.} We also demonstrate the  promising numerical performance of the proposed method on test problems. 

\subsection{Notation and preliminaries}
Without specification, $\|\cdot\|$ denotes the Euclidean norm for vectors
and the spectral operator norm for matrices. 
%\red{For a differentiable matrix-valued mapping $H:\mathbb R^n\to\mathbb R^{p\times q}$, the norm of its derivative is defined by
% \[
% \|DH(x)\|:=\sup_{\substack{\|d\|=1,\,d\in\mathbb R^n}}\|DH(x)[d]\|,
% \]
%where the matrix $DH(x)[d]$ is measured in the spectral norm.} 
For any matrix $A$, $\mu_i(A)$ (resp. $\mu_{\min}(A)$) denotes its $i$-th (resp. smallest) singular value. Denote $\nabla c(x)=(\nabla c_1(x),\ldots,\nabla c_m(x))$ and 
$\nabla C(x;\xi)=(\nabla C_1(x;\xi),\ldots,\nabla C_m(x;\xi))$. %, with $m\le n$. 
%\deleted{Additionally, \(\nabla^2 c(x)\) denotes the derivative \(D(\nabla c)(x)\) of the Jacobian mapping \(\nabla c\).} 
{Additionally, for each $i=1,\ldots,m$, let $\nabla^2 c_i(x)\in\R^{n\times n}$ denote the Hessian of $c_i(x)$. For any $d\in\R^n$, we define $\nabla^2 c(x)[d]:=(\nabla^2 c_1(x)d,\ldots,\nabla^2 c_m(x)d)\in\R^{n\times m}$, and $\|\nabla^2 c(x)\|:=\sup_{\substack{\|d\|=1}}\|\nabla^2 c(x)[d]\|$.} For any $k \ge 0$, let $\mathcal{F}_k$ denote the $\sigma$-algebra generated by all randomness up to iteration $k-1$. We use $\mathbb{P}_k(\cdot) := \mathbb{P}(\cdot |\mathcal{F}_k)$ and $\mathbb{E}_k[\cdot] := \mathbb{E}[\cdot|\mathcal{F}_k]$ to denote the conditional probability and conditional expectation with respect to $\mathcal{F}_k$, respectively. 

We next lay out the assumptions and give the definition of approximate solutions of problem (\ref{p}).
\begin{assumption}\label{a1}
Let $\mathcal X$ be an open convex set that contains {all realizations of} {$\{x_k\}$} generated by the associated algorithm,  $f$ is lower bounded by {a finite value} $f_{\rm low}$, and there  exist constants $G,M>0$ such that $\|\nabla f(x)\|\leq G$ and $\|c(x)\|\le M$ for any $x\in\mathcal X$.
\end{assumption}

\begin{assumption}\label{a2}
For almost every $\xi\in\Xi$, functions $F(\cdot,\xi)$ and $C(\cdot,\xi)$ are differentiable, and for any $x\in\mathcal{X}$, it holds that 
\[\E_\xi[\nabla F(x,\xi)] = \nabla f(x),\quad \E_\xi[C(x,\xi)] = c(x),\quad\E_\xi[\nabla C(x,\xi)] = \nabla c(x).\]
There exist {$\sigma_f, L_f>0$} such that for any $x,y\in\mathcal{X}$,
\begin{align*}
    \E_\xi[\|\nabla F(x,\xi) - \nabla f(x)\|^2] &\leq \sigma_f^2,\quad \E_\xi[\|\nabla F(x;\xi) - \nabla F(y;\xi)\|^2] \leq L_f^2 \|x-y\|^2.
\end{align*}
There exist {$\sigma_c,\sigma_J,L_J, L_c>0$} such that for any $x,y\in \mathcal{X}$,
\begin{align*}
  \E_\xi[\|\nabla C(x,\xi) - \nabla c(x)\|^2] &\leq \sigma_J^2,\quad \E_\xi[\|\nabla C(x;\xi) - \nabla  C(y;\xi)\|^2]\leq L_J^2\|x-y\|^2,\\
    \E_\xi[\| C(x,\xi) -  c(x)\|^2]& \leq \sigma_c^2,\quad\, \E_\xi[\|  C(x;\xi) -  C(y;\xi)\|^2] \leq L_c^2\|x-y\|^2.
\end{align*}
\end{assumption}

By Jensen's inequality, Assumption \ref{a2} implies that $\nabla f$, $\nabla c$, and $c$ are Lipschitz continuous on $\mathcal X$ with constants $L_f$, $L_J$, and $L_c$, respectively.  Consequently, $\|\nabla c(x)\|\le L_c$ and $\|\nabla^2 c(x)\|\le L_J$ for all $x\in\mathcal X$.
%\be\begin{aligned}\label{Lip-f-c}
%&\|\nabla f(x)-\nabla f(y)\|\le L_f\|x-y\|,\   \|\nabla  c(x)-\nabla c(y)\|\le L_J\|x-y\|, \\
%&\| c(x)- c(y)\|\le L_c\|x-y\|,
%\end{aligned}\ee

\begin{assumption}\label{a3}
There exist positive constants $L_h^f$ and $L_h^c$ such that 
 \begin{align*}
\|\nabla^2 f(x)-\nabla^2 f(y)\|\le L_h^f\|x-y\|,\quad \|\nabla^2 c(x)-\nabla^2 c(y)\|\le L_h^c\|x-y\|,\quad \forall x,y \in \mathcal X.
\end{align*}
\end{assumption}

\begin{assumption}\label{a4}
There exists a positive constant \(\nu\) such that for any $x\in\mathcal X$, the singular values of \(\nabla c(x)\) are bounded below by \(\nu\).
\end{assumption}

%\begin{definition}
   Given $\epsilon>0$, 
   %a point $x\in\R^n$ is called an $\epsilon$-KKT point of \eqref{p}, if there exists $\lambda\in\R^m$ such that  $\|\nabla f(x) + \nabla c(x) \lambda\|\le \epsilon$ and $\|c(x)\|\le \epsilon.$
	a point $x\in \R^n$ is called a {\it stochastic $\epsilon$-KKT point}  of \eqref{p}, if 
    {\begin{align}\label{def-lambda}
        \lambda(x)=-(\nabla c(x))^{\dagger}\nabla f(x),
    \end{align}
    and}
	\begin{align}\label{def-KKT}
	    \E[\|\nabla f(x) +  \nabla c(x)\lambda(x)\|]   \leq \epsilon \quad\mbox{and}\quad \E[\|c(x)\|]  \leq \epsilon,
	\end{align}
    {where the expectation is taken with respect to all the random variables generated when producing $x$.}
%\end{definition}

\section{Algorithm Framework}
In this section, we present a single-loop stochastic first-order method for solving \eqref{p}. The framework is built on a decomposition strategy for constructing the search direction, together with adaptive updates of the merit parameter and the stepsize. We begin with the deterministic counterpart of problem \eqref{p}, where both the objective and the constraints are deterministic, and consider the {computation of a} search direction. The main idea is to construct, at each iterate, a search direction through a decomposition into a tangential component for objective reduction and a normal component for feasibility improvement. More precisely, define the tangent space at a point $x$ by
\[
V(x):=\{v\in\mathbb{R}^n:\nabla c(x)^\top v=0\}.
\]
Projecting $\nabla f(x)$ onto $V(x)$ gives the tangential direction that reduces the objective while remaining in the tangent space of the constraint manifold. Under Assumption \ref{a4}, the projection is given by
\begin{align}\label{eq:proj_auxtan}
    {\rm P}_{V(x)} \nabla f(x) = \argmin_{v \in V(x)} \frac{1}{2} \|v - \nabla f(x)\|^2 = \nabla f(x) - (\nabla c(x)^{\top} )^{\dagger} \nabla c(x)^{\top} \nabla f(x).
\end{align}
{With the definition of the mapping $\lambda(x)$ in (\ref{def-lambda}), it is easy to obtain $ \nabla f(x) + \nabla c(x)\lambda(x)={\rm P}_{V(x)}\nabla f(x)$.} To improve feasibility, we use the correction direction $-\nabla c(x)c(x)$, which is the steepest descent direction for the squared constraint violation $\frac12\|c(x)\|^2$. Combining this component with the tangential direction gives
\[
s(x)=-{\rm P}_{V(x)}\nabla f(x)-w\nabla c(x)c(x),
\]
where $w\in(0,1)$. %The following lemma shows that $\|{\rm P}_{V(x)} \nabla f(x)\|$ together with $\|c(x)\|$ provides a natural criterion for the $\epsilon$-KKT conditions.
%\begin{lemma}\label{lm:OF_crit}
%Under Assumption \ref{a4}, define the multiplier mapping 
%\[\lambda(x) := -(\nabla c(x)^{\top}\nabla c(x))^{-1}\nabla c(x)^{\top}\nabla f(x).\] Then, we have
%$\|\nabla f(x)+\nabla c(x)\lambda(x)\| = \|{\rm P}_{V(x)}\nabla f(x)\|.$
%Thus, if
%$\|{\rm P}_{V(x)}\nabla f(x)\|+\|c(x)\|\le \epsilon$
%for some $\epsilon>0$, then $x$ is an $\epsilon$-KKT point. Moreover, if 
%\begin{align*}
%    \mathbb E[\|{\rm P}_{V(x)}\nabla f(x)\|+\|c(x)\|]\le \epsilon,
%\end{align*}
%then $x$ is a stochastic $\epsilon$-KKT point.
%\end{lemma}

{For problem (\ref{p}),} since the exact information is computationally intractable in stochastic settings, we instead work with the stochastic estimators $\tilde\nabla f_k$, $\tilde\nabla c_k$, and $\tilde c_k$ at iteration $k$ to approximate the true objective gradient $\nabla f(x_k)$, constraint gradient $\nabla c(x_k)$, and constraint value $c(x_k)$, respectively. For notational simplicity, we abbreviate $\nabla f(x_k),\nabla c(x_k),c(x_k)$ as $\nabla f_k,\nabla c_k,c_k$ throughout the remainder of this paper. The search direction is then defined as
\begin{equation}\label{s_k}
    \tilde{s}_k = -{\rm P}_{\widetilde V_k}\tilde{\nabla} f_k - w\tilde{\nabla} c_k \tilde c_k,
\end{equation}
where $w \in (0,1)$ is a fixed parameter and $\widetilde V_k := \{v \in \mathbb{R}^n : \tilde\nabla c_k^\top v = 0\}$ represents the null space associated with the stochastic Jacobian. The operator 
\[
    {\rm P}_{\widetilde V_k}\tilde{\nabla} f_k := \tilde{\nabla} f_k - (\tilde \nabla c_k^\top)^{\dagger}\tilde \nabla c_k^\top\tilde{\nabla} f_k
\]
denotes the orthogonal projection of $\tilde{\nabla} f_k$ onto $\widetilde V_k$, computed via the Moore-Penrose pseudoinverse. Consequently, the two components of $\tilde s_k$ are orthogonal by definition so that ${\rm P}_{\widetilde V_k}\tilde{\nabla} f_k \perp \tilde{\nabla} c_k \tilde c_k$. 
This decomposition shares a similar structure with {stochastic} SQP-type directions, as both combine a tangential component for objective reduction with a normal component for feasibility improvement. Specifically, there exists $y_k\in\R^m$ such that $(\tilde s_k,y_k)$ satisfies the following linear system:
\be\begin{aligned}\label{sqp}
    \begin{bmatrix}
    I_n & \tilde\nabla c_k \\
    \tilde\nabla c_k^\top & 0
    \end{bmatrix}
    \begin{bmatrix}
    \tilde s_k \\
    y_k
    \end{bmatrix}
    =
    -
    \begin{bmatrix}
    \tilde\nabla f_k \\
    w \tilde \nabla c_k^\top\tilde\nabla c_k\tilde c_k
    \end{bmatrix}.
    \end{aligned}\ee
In standard SQP methods, the matrix in place of $I_n$ is typically an approximation of the Hessian of the Lagrangian, and the corresponding linear system arises from a quadratic subproblem with linearized constraints, whereas \eqref{sqp} incorporates the damped normal correction $-w\tilde{\nabla}c_k\tilde c_k$ to reduce the sampled constraint violation.

To measure the progress along the search direction, we use Fletcher's augmented Lagrangian as the merit function. It provides a smooth model for constrained optimization and is defined by
\begin{align}\label{ALM}
    \calL(x,\rho)=f(x)+\langle \lambda(x),c(x)\rangle+\frac{\rho}{2}\|c(x)\|^2,
\end{align}
where $\lambda(x)=-(\nabla c(x))^{\dagger}\nabla f(x)$ and \(\rho>0\) is a merit parameter. {It is worth noting that Fletcher's augmented Lagrangian introduces the multiplier mapping $\lambda(x)$, whose evaluation amounts to solving the $m$-dimensional normal equations associated with a least-square problem, at a cost comparable to that of tangent projection.}
%This computation is of the same order as tangent projection, and $\calL(x,\rho)$ is used only to assess progress along $\tilde s_k$, rather than to be minimized for generating the next iterate.

The following proposition  states  the Lipschitz continuity of the multiplier mapping $\lambda(\cdot)$ and the smoothness of the merit function {in $x$}.

\begin{proposition}\label{le:lambda}
Under Assumptions \ref{a1}-\ref{a4}, the mapping $\lambda(\cdot)$ and its Jacobian $\nabla \lambda(\cdot)$ are both Lipschitz continuous with constants \(L_\lambda\) and \(L_{\lambda}^1\), respectively,  and \(\nabla \mathcal L(x,\rho)\) is \(L_\rho\)-Lipschitz continuous in \(x\)  for any $\rho>0$, where
\begin{align*}
    L_\lambda &:= \frac{1}{\nu^2} \Big( \frac{2GL_c^2 L_J}{\nu^2} + G L_J + L_c L_f \Big),\\
    L^1_\lambda &:= \frac{8GL_J^2L_c^3}{\nu^6}
+\frac{2L_c\big(2L_JL_cL_f+3GL_J^2+GL_cL_h^c\big)}{\nu^4}
+\frac{2L_JL_f+L_h^cG+L_cL_h^f}{\nu^2},\\
L_\rho&:=L_f+2L_\lambda L_c+ML_\lambda^1+GL_cL_J/{\nu^2}+\rho (L_c^2+ML_J).
\end{align*}
\end{proposition}
\begin{proof}
    See Appendix \ref{app:proof}.
\end{proof}

{To derive a desired iterative progress regarding the merit function, we choose the step size along the search direction as 
\begin{align}\label{eq:eta-update}
    \eta_k=\frac{1}{4L_k} \mbox{ with } L_k:=L_{\rho_k}, 
\end{align}
for each $k\geq 1$, where the merit parameter is determined through an adaptive scheme:}
   \begin{align}\label{eq:rho-update}
    \rho_k
    =\max\left\{
    \frac{4L_\lambda^2+2w^2L_c^2+2}{w\chi_k},
    \rho_{k-1}
    \right\}
    \text{ with }\chi_k:=
    \begin{cases}
    \mu_{\min}(\tilde\nabla c_k)^2, &\text{if } \mu_{\min}(\tilde\nabla c_k)\ge \frac{\nu}{2},\\[1mm]
    2\nu^2, &\text{otherwise}.
    \end{cases}
\end{align}
%where $L_\lambda$ is the Lipschitz constant of the multiplier mapping $\lambda(\cdot)$ introduced in Lemma \ref{le:lambda}, and $L_k$ is defined in Lemma \ref{le:calL}.

{A remaining issue to complete the algorithm is regarding the construction of} the stochastic estimators \(\tilde\nabla f_k\), \(\tilde\nabla c_k\), and \(\tilde c_k\). Motivated by the boundedness assumptions on $\nabla f(x)$, $\nabla c(x)$ and $c(x)$, we employ a truncated SPIDER-type \cite{FLLZ2018} technique to approximate the objective gradient, constraint Jacobian, and constraint value. Specifically, given $r > 0$, let $\mathbb B(r)$ denote the closed ball centered at the origin with radius $r$ in the relevant vector or matrix space, and let {${\rm P}_{\mathbb B(r)}$} denote the projection onto $\mathbb B(r)$, taken with respect to the Euclidean norm for vectors and the Frobenius norm for matrices. For a prescribed epoch length $\tau \in \mathbb{N}$, define $q(k) := \lfloor (k-1)/\tau \rfloor \tau+1$ as the start of the epoch to which iteration $k$ belongs. Then, at the $k$-th iteration {with $k\geq1$}, the estimators are updated as follows{: if} $k=q(k)$, 
\be\begin{aligned}\label{eq:spider0}
\tilde \nabla f_k&={\rm P}_{\mathbb B(G)}
\Big(\frac{1}{B_k^f}\sum_{\xi\in \mathcal{B}_k^f}\nabla F(x_k,\xi)\Big),\quad
\tilde \nabla c_k={\rm P}_{\mathbb B(L_c)}
\Big(\frac{1}{B_k^J}\sum_{\xi\in \mathcal{B}_k^J}\nabla C(x_k,\xi)\Big),\\
\tilde c_k&={\rm P}_{\mathbb B(M)}\Big(\frac{1}{B_k^c}\sum_{\xi\in \mathcal{B}_k^c} C(x_k,\xi)\Big);
\end{aligned}\ee
if $k>q(k)$, 
\begin{align}\label{eq:spider-k}
\tilde \nabla f_k&={\rm P}_{\mathbb B(G)}\Big(\tilde \nabla f_{k-1}+\frac{1}{B_k^f}\sum_{\xi\in \mathcal{B}_k^f}\big(\nabla F(x_k,\xi)-\nabla F(x_{k-1},\xi)\big)\Big),\notag\\
\tilde \nabla c_k&={\rm P}_{\mathbb B(L_c)}\Big(\tilde \nabla c_{k-1}+\frac{1}{B_k^J}\sum_{\xi\in \mathcal{B}_k^J}\big(\nabla C(x_k,\xi)-\nabla C(x_{k-1},\xi)\big)\Big),\\
\tilde c_k&={\rm P}_{\mathbb B(M)}\Big(\tilde c_{k-1}+\frac{1}{B_k^c}\sum_{\xi\in \mathcal{B}_k^c}\big(C(x_k,\xi)-C(x_{k-1},\xi)\big)\Big),\notag
\end{align}
where $\mathcal B_k^f$, $\mathcal B_k^J$ and $\mathcal B_k^c$, $k\ge1$ are independently and randomly generated sample sets with $|\mathcal B_k^f| = B_k^f$, $|\mathcal B_k^J| = B_k^J$ and $|\mathcal B_k^c| = B_k^c$. {We are now ready to present }the complete algorithmic framework is summarized in Algorithm \ref{alg1}.

\begin{algorithm}[ht]
  \caption{ } 
  \label{alg1}
  \begin{algorithmic}[1] 
    \REQUIRE Initial iterate $x_1$, initial merit parameter $\rho_0$, positive integers $K$ and $\tau$.
    \FOR{$k=1,\ldots,K$ }
    \IF{{$(k-1) \text{ mod } \tau =0$}}
		\STATE Compute $\tilde\nabla f_k$, $\tilde\nabla c_k$ and $\tilde c_k$ through (\ref{eq:spider0}).
		\ELSE
		\STATE Compute $\tilde\nabla f_k$, $\tilde\nabla c_k$ and $\tilde c_k$ through (\ref{eq:spider-k}).
	\ENDIF
    \STATE Compute \(\tilde{s}_k\) via \eqref{s_k}. 
    \STATE Compute \(\eta_k\) via \eqref{eq:eta-update} with \(\rho_k\)  computed by \eqref{eq:rho-update}.
    \STATE Update \(x_{k+1} = x_k + \eta_k \tilde{s}_k\).
    \ENDFOR 
    \ENSURE $x_{R}$ where $R\in\{1,\ldots,K\}$ is uniformly chosen at random.
  \end{algorithmic} 
\end{algorithm}

For the analysis of the SPIDER-type estimators, we define the epoch-start filtration $\bar{\mathcal F}_k := \mathcal F_{q(k)}$, along with the conditional probability $\bar{\mathbb P}_k(\cdot) := \mathbb P(\cdot \mid \bar{\mathcal F}_k)$ and expectation $\bar{\mathbb E}_k[\cdot] := \mathbb E[\cdot \mid \bar{\mathcal F}_k]$. By definition, {it holds that} $\bar{\mathcal F}_k \subseteq \mathcal F_k$ for all $k \ge 1$, where equality holds exactly {when} $k = q(k)$. Conditioning on $\bar{\mathcal F}_k$ fixes the information available at the beginning of the epoch, which facilitates bounding the accumulated estimation error over that epoch. {To faciliate the subsequent theoretical analysis, we first impose the following assumption on stochastic estimators. Eventaully, we will realize this assumption through proper sampling strategies.} 
\begin{assumption}\label{a5}
    There exist positive constants $\{\tilde\sigma_k\}_{k\ge 1}\subset (0,1)$ such that for any $k\ge1$,
    \be\label{error-bnd}
     \bar\E_k[\|\tilde \nabla f_k - \nabla f_k\|^2] \leq \tilde\sigma_k^2,\ \ ~\bar\E_k[\|\tilde \nabla c_k - \nabla c_k\|^2] \leq  \tilde\sigma_k^2,\ \ ~\bar\E_k[\|\tilde c_k - c_k\|^2] \leq \tilde\sigma_k^2.
    \ee
\end{assumption}

\section{Theoretical Analysis} \label{sec:ana}
%\subsection{Preliminaries for the analysis}
%Before proceeding to the main theoretical analysis, this subsection establishes two essential properties, 

\subsection{Directional Perturbation Bounds}
In this subsection, we derive a perturbation bound for the stochastic search direction $\tilde s_k$ with respect to its deterministic counterpart $s_k$. The main difficulty is that, although the true Jacobian satisfies a uniform lower singular value bound under Assumption \ref{a4}, the stochastic Jacobian may not remain uniformly nondegenerate at every iteration. To handle this issue, we partition the analysis into the good event and the bad event according to whether the stochastic Jacobian preserves a uniform lower singular value. Specifically, we define the good event
\[
\mathcal{G}_k:=\left\{\mu_{\min}(\tilde\nabla c_k)\ge \frac{\nu}{2}\right\}.
\]
The following lemma provides a conditional probability bound for the corresponding bad event \(\mathcal{G}_k^c\).

\begin{lemma}\label{lem:bad-event}
Under Assumptions \ref{a4} and \ref{a5}, it holds that
%\begin{equation}\label{eq:bad-event-prob}
\(\bar{\mathbb{P}}_k(\mathcal{G}_k^c)\le \frac{4}{\nu^2}\tilde\sigma_{k}^2.\)
%\end{equation}
\end{lemma}

\begin{proof}
It follows from Weyl's inequality (Lemma \ref{lem:weyl}) that
\[\mu_{\min}(\tilde{\nabla} c_k)\ge\mu_{\min}(\nabla c_k)-\|\tilde{\nabla} c_k-\nabla c_k\|.\]
Then using Assumption \ref{a4}, we have $\mu_{\min}(\tilde{\nabla} c_k)\ge\nu-\|\tilde{\nabla} c_k-\nabla c_k\|$. Therefore, it holds that
\[
\mathcal{G}_k^c
=
\left\{\mu_{\min}(\tilde{\nabla} c_k)<\frac{\nu}{2}\right\}
\subseteq
\left\{\|\tilde{\nabla} c_k-\nabla c_k\|\ge \frac{\nu}{2}\right\}.
\]
Applying Markov's inequality and Assumption~\ref{a5} yields
\begin{align*}
    \bar{\mathbb{P}}_k(\mathcal{G}_k^c)\le
    \bar{\mathbb{P}}_k\left(\|\tilde{\nabla} c_k-\nabla c_k\|^2\ge \frac{\nu^2}{4}\right)
   \le\frac{4}{\nu^2}\,\bar{\mathbb{E}}_k[\|\tilde{\nabla} c_k-\nabla c_k\|^2]\le\frac{4}{\nu^2}\tilde\sigma_{k}^2.
\end{align*}
This completes the proof.
\end{proof}

We now proceed to establish a perturbation bound between the stochastic search direction $\tilde s_k$ and the {true} direction $s_k$. %In what follows, we adopt the convention that an inequality stated on an event $\mathcal A$ is understood after multiplying both sides by the indicator $\mathbf 1_{\mathcal A}$ and taking the expectation.
\begin{lemma}\label{le:s-error}
     Suppose that Assumptions \ref{a1}-\ref{a2}, {\ref{a4}, and \ref{a5}} hold. Then for any $k\geq1$, it holds that
    \begin{align}\label{sk-sk}
       \bar{\mathbb{E}}_k[\|s_k-\tilde s_k\|^2]&\leq C_s^2\tilde\sigma_k^2,
    \end{align}
    where {$C_s:=[2C_p^2+4w^2(M^2+L_c^2)+32(G^2+w^2M^2L_c^2)/{\nu^2}]^{\frac{1}{2}}$ with $C_p:=[5(1+(17G^2L_c^2+16L_c^4)/\nu^4+64G^2L_c^6/\nu^8)]^{\frac{1}{2}}$}.
\end{lemma}
\begin{proof}
    On the good event $\mathcal{G}_k$, after localization by the indicator $\mathbf 1_{\mathcal G_k}$, the Sherman–Morrison–Woodbury formula (Lemma \ref{lm:smw}) implies that
    \begin{align*}
        &\bar{\E}_k[\| (\tilde \nabla c_k^\top \tilde \nabla c_k)^{-1} - (\nabla c_k^\top \nabla c_k)^{-1} \|^2\mathbf 1_{\mathcal G_k}]\\
        &=\bar{\E}_k\Big[\big\|(\tilde \nabla c_k^{\top} \tilde \nabla c_k)^{-1}\big[\nabla c_k^{\top} \nabla c_k - \tilde \nabla c_k^{\top} \tilde \nabla c_k\big](\nabla c_k^{\top} \nabla c_k)^{-1}\big\|^2\mathbf 1_{\mathcal G_k}\Big]\\
        &\leq \bar{\E}_k\Big[\big\|(\tilde \nabla c_k^{\top} \tilde \nabla c_k)^{-1}\big\|^2 \big\|\nabla c_k^{\top} \nabla c_k - \tilde \nabla c_k^{\top} \tilde \nabla c_k\big\|^2\big\|(\nabla c_k^{\top} \nabla c_k)^{-1}\big\|^2\mathbf 1_{\mathcal G_k}\Big]\\
        &\leq \frac{32}{\nu^8} \bar{\E}_k\big[(\|\nabla c_k\|^2+\|\tilde \nabla c_k\|^2)\|\tilde \nabla c_k - \nabla c_k\|^2\mathbf 1_{\mathcal G_k}\big],
    \end{align*}
    where the last inequality comes from the definition of $\mathcal{G}_k$ and Assumption \ref{a4}.
    Under Assumptions \ref{a1}-\ref{a2} and \ref{a5}, together with the truncation bound $\|\tilde\nabla c_k\| \le L_c$, we obtain
   \be\begin{aligned}\label{pv-pv}
        &\bar{\E}_k[\|{\rm P}_{V_k}\nabla f_k-{\rm P}_{\widetilde V_k}\tilde{\nabla} f_k\|^2\mathbf 1_{\mathcal G_k}]\\
        &=\bar{\E}_k[\|\nabla f_k-\nabla c_k(\nabla c_k^\top\nabla c_k)^{-1}\nabla c_k^\top\nabla f_k-\tilde{\nabla} f_k+\tilde\nabla c_k(\tilde\nabla c_k^\top\tilde\nabla c_k)^{-1}\tilde\nabla c_k^\top\tilde\nabla f_k\|^2\mathbf 1_{\mathcal G_k}]\\
        &\leq5\bar{\E}_k[\|\nabla f_k-\tilde{\nabla} f_k\|^2\mathbf 1_{\mathcal G_k}]+5\bar{\E}_k[\|\nabla c_k-\tilde\nabla c_k\|^2\|(\nabla c_k^\top\nabla c_k)^{-1}\nabla c_k^\top\nabla f_k\|^2\mathbf 1_{\mathcal G_k}]\\
        &\quad+5\bar{\E}_k[\|\tilde\nabla c_k\|^2\|(\nabla c_k^\top\nabla c_k)^{-1}-(\tilde\nabla c_k^\top\tilde\nabla c_k)^{-1}\|^2\|\nabla c_k^\top\nabla f_k\|^2\mathbf 1_{\mathcal G_k}]\\
        &\quad+5\bar{\E}_k[\|\tilde\nabla c_k(\tilde\nabla c_k^\top\tilde\nabla c_k)^{-1}\|^2\|\tilde\nabla c_k-\nabla c_k\|^2\|\nabla f_k\|^2\mathbf 1_{\mathcal G_k}]\\
        &\quad+5\bar{\E}_k[\|\tilde\nabla c_k(\tilde\nabla c_k^\top\tilde\nabla c_k)^{-1}\tilde\nabla c_k^\top\|^2\|\tilde\nabla f_k-\nabla f_k\|^2\mathbf 1_{\mathcal G_k}]\\
        &\leq 5\Big(1+\frac{G^2L_c^2}{\nu^4}+\frac{64G^2L_c^6}{\nu^8}+\frac{16G^2L_c^2}{\nu^4}+\frac{16L_c^4}{\nu^4}\Big)\tilde\sigma_k^2:= C_p^2\tilde\sigma_k^2.
    \end{aligned}\ee
   Then appealing to the Cauchy-Schwarz inequality and the truncation bound $\|\tilde c_k\| \le M$, we arrive at
    \begin{align*}
        &\bar{\E}_k[\|\nabla c_k  c_k -\tilde \nabla c_k \tilde c_k\|^2\mathbf 1_{\mathcal G_k}]\\
        &\leq2\bar{\E}_k[\|\nabla c_k\|^2\|\tilde c_k-c_k\|^2\mathbf 1_{\mathcal G_k}]+2\bar{\E}_k[\|\tilde\nabla c_k-\nabla c_k\|^2\|\tilde c_k\|^2\mathbf 1_{\mathcal G_k}]\leq 2(M^2+L_c^2)\tilde\sigma_k^2.
    \end{align*}
  The contribution from the good event is therefore bounded by
    \begin{align*}
        \bar{\E}_k[\|s_k-\tilde{s}_k\|^2\mathbf{1}_{\mathcal{G}_k}]&\leq2\bar{\E}_k[\|{\rm P}_{V_k}\nabla f_k-{\rm P}_{\widetilde V_k}\tilde{\nabla} f_k\|^2\mathbf{1}_{\mathcal{G}_k}]+2w^2\bar{\E}_k[\|\nabla c_k  c_k - \tilde \nabla c_k \tilde c_k\|^2\mathbf{1}_{\mathcal{G}_k}]\\
        &\leq(2C_p^2+4w^2M^2+4w^2L_c^2)\tilde\sigma_k^2.
    \end{align*}
    On the bad event $\mathcal{G}_k^c$, it follows from Assumptions \ref{a1}-\ref{a2} and the truncation bounds that
    \begin{align}\label{Pv-bad}
       \bar{\E}_k[\|s_k-\tilde{s}_k\|^2\mathbf{1}_{\mathcal{G}_k^c}]&\leq2\bar{\E}_k[\|{\rm P}_{V_k}\nabla f_k-{\rm P}_{\widetilde V_k}\tilde{\nabla} f_k\|^2\mathbf{1}_{\mathcal{G}_k^c}]+2w^2\bar{\E}_k[\|\nabla c_k  c_k - \tilde \nabla c_k \tilde c_k\|^2\mathbf{1}_{\mathcal{G}_k^c}]\notag\\
       &\leq(8G^2+8w^2M^2L_c^2)\bar{\mathbb{P}}_k(\mathcal{G}_k^c).
    \end{align}
   Combining the bounds derived on $\mathcal{G}_k$ and $\mathcal{G}_k^c$ with the probability estimate in Lemma \ref{lem:bad-event}, we derive
    \begin{align*}
       \bar{\mathbb{E}}_k [\|\tilde s_k-s_k\|^2]
     &=\bar{\mathbb{E}}_k\!\left[\|\tilde s_k-s_k\|^2 \mathbf{1}_{\mathcal{G}_k}\right]+\bar{\mathbb{E}}_k\!\left[\|\tilde s_k-s_k\|^2 \mathbf{1}_{\mathcal{G}_k^c}\right]\\
     &\leq  (2C_p^2+4w^2M^2+4w^2L_c^2)\tilde\sigma_k^2+\frac{32(G^2+w^2M^2L_c^2)}{\nu^2}\tilde\sigma_{k}^2:=C_s^2\tilde\sigma_k^2,
    \end{align*}
    which completes the proof.
\end{proof}

\subsection{Convergence analysis}
In this subsection, we study the convergence properties of Algorithm \ref{alg1}. We first observe that the update rule \eqref{eq:rho-update} guarantees a uniform upper bound on the merit parameter. On the good event $\mathcal G_k=\{\mu_{\min}(\tilde\nabla c_k)\ge \nu/2\}$, we have $\mu_{\min}(\tilde\nabla c_k)^2 \ge \nu^2/4$. On the complementary event $\mathcal G_k^c$, the denominator is trivially bounded below by the constant $2w\nu^2$. Consequently, the candidate update is always bounded above by $4\bigl(4L_\lambda^2+2(w^2L_c^2+1)\bigr)/(w\nu^2)$. Since the sequence $\{\rho_k\}$ is non-decreasing, it follows immediately by induction that
\begin{align*}
    {\rho_0\leq\,}\rho_k \le \rho_{\max} := \max\left\{\rho_0, \frac{4\bigl(4L_\lambda^2+2(w^2L_c^2+1)\bigr)}{w\nu^2}\right\}, \quad \forall k \ge 1.
\end{align*}
As a direct consequence, \(L_k\) is uniformly bounded from below and above. More precisely, let
{\[L_{\min}:=L_{\rho_0}\quad \text{and} \quad L_{\max}:=L_{\rho_{\max}},\]
with $L_\rho$ defined in Proposition~\ref{le:lambda}.}
%\(L_{\min}\) and \(L_{\max}\) denote the values of \(L_k\) associated with $\rho_0$ and $\rho_{\max}$, respectively.
Given that $\rho_k \in [\rho_0, \rho_{\max}]$, the update rule \eqref{eq:eta-update} directly yields
\[\eta_{\min}:=\frac{1}{4L_{\max}}
\le\eta_k
\le\frac{1}{4L_{\min}}
=:\eta_{\max},\quad \forall k\ge 1.\]
Equipped with these uniform bounds, we are now ready to characterize the convergence behavior of the sequence generated by Algorithm \ref{alg1}. We first establish a descent property of $\calL(x,\rho)$, which plays a central role in the subsequent analysis.

\begin{lemma}
    Suppose Assumptions {\ref{a1}-\ref{a4} and \ref{a5}} hold. Then for any $k\geq1$, we have
    \be\begin{aligned}\label{eq:descent-L}
 \E[\calL(x_{k+1},\rho_k)]
 &\leq\E[\calL(x_k,\rho_k)]-\E[\eta_k(\|{\rm P}_{V_k}{\nabla} f_k\|^2+(w^2L_c^2+1)\|c_k\|^2)]\\
 &\quad+\E\big[\big(\eta_k\|{\rm P}_{V_k}{\nabla} f_k\|+L_c\eta_k\rho_k\|c
        _k\|\big)(\|\Delta P_k\|+w\|\Delta c_k\|)\big]\\
        &\quad+wM\E[\eta_k\rho_k \| c_k \|\|\Delta J_k\|^2]+\frac{1}{4}\E[(2\eta_k^2L_k+\eta_k)\|\tilde{s}_k\|^2],
\end{aligned}\ee
where $\Delta P_k:={\rm P}_{\widetilde V_k}\tilde{\nabla} f_k-{\rm P}_{V_k}{\nabla} f_k$, $\Delta c_k := \tilde\nabla c_k \tilde c_k - \nabla c_k c_k$ and $\Delta J_k:=\tilde\nabla c_k-\nabla c_k$.
\end{lemma}
\begin{proof}
    It follows from the Lipschitz continuity of $\nabla \calL(x,\rho)$ {in $x$} established in Proposition \ref{le:lambda} and $x_{k+1}=x_k+\eta_k\tilde{s}_k$ that 
\begin{align}\label{L-expand}
\E[\calL(x_{k+1},\rho_k)]&\leq\E[\calL(x_k,\rho_k)]+\E[\langle\nabla\calL(x_k,\rho_k),x_{k+1}-x_k\rangle]+\frac{1}{2}\E[{\eta_k^2L_k}\|\tilde{s}_k\|^2]\notag\\
    &=\E[\calL(x_k,\rho_k)]+\E[\langle\nabla\calL(x_k,\rho_k),\eta_k\tilde{s}_k\rangle]+\frac{1}{2}\E[{\eta_k^2L_k}\|\tilde{s}_k\|^2],
\end{align}
For the second term, one has 
\[ -\E[\langle\nabla\calL(x_k,\rho_k),\eta_k\tilde{s}_k\rangle]=\E[\eta_k\langle\nabla\calL(x_k,\rho_k),{\rm P}_{\widetilde V_k}\tilde{\nabla} f_k + w\tilde{\nabla} c_k \tilde c_k\rangle]\]
from (\ref{s_k}). Recalling the definition of $\Delta P_k$, we have
    \begin{align*}
       &\langle\nabla\calL(x_k,\rho_k),{\rm P}_{\widetilde V_k}\tilde{\nabla} f_k\rangle\\
       &=\langle{\nabla} f_k,{\rm P}_{\widetilde V_k}\tilde{\nabla} f_k\rangle+\langle\nabla\lambda_kc_k,{\rm P}_{\widetilde V_k}\tilde{\nabla} f_k\rangle+\langle\lambda_k,\nabla c_k^\top{\rm P}_{\widetilde V_k}\tilde{\nabla} f_k\rangle+\rho_k\langle\nabla c_kc_k,{\rm P}_{\widetilde V_k}\tilde{\nabla} f_k\rangle\\
       &=\langle{\nabla} f_k,{\rm P}_{\widetilde V_k}\tilde{\nabla} f_k\rangle+\langle\nabla\lambda_kc_k,{\rm P}_{\widetilde V_k}\tilde{\nabla} f_k\rangle+\langle\lambda_k,\nabla c_k^\top\Delta P_k\rangle+\rho_k\langle\nabla c_kc_k,\Delta P_k\rangle,
    \end{align*}
    where the second equality is due to the fact that {\({\rm P}_{V_k}\nabla f_k\) is orthogonal to the column space of $\nabla c_k$}. With the notation $\Delta c_k := \tilde\nabla c_k \tilde c_k - \nabla c_k c_k$, it is easy to show that
    \begin{align*}
        &\langle\nabla\calL(x_k,\rho_k),w\tilde\nabla c_k\tilde c_k\rangle\\
        &=\langle{\nabla} f_k,w\tilde\nabla c_k\tilde c_k\rangle+\langle\nabla\lambda_kc_k,w\tilde\nabla c_k\tilde c_k\rangle+\langle\nabla c_k\lambda_k,w\tilde\nabla c_k\tilde c_k\rangle+\rho_k\langle\nabla c_kc_k,w\tilde\nabla c_k\tilde c_k\rangle\\
        &=\langle{\nabla} f_k-\nabla c_k(\nabla c_k^\top\nabla c_k)^{-1}\nabla c_k^\top\nabla f_k,w\tilde\nabla c_k\tilde c_k\rangle+\rho_k w\langle\nabla c_kc_k,\Delta c_k\rangle\\
        &\quad+w\langle\nabla\lambda_k\tilde\nabla c_k\tilde c_k,c_k\rangle+\rho_k w\|\nabla c_k c_k\|^2.
    \end{align*}
  Using the two preceding relations, we deduce that
    \be\begin{aligned}\label{L-times-s}
        &-\E[\langle\nabla\calL(x_k,\rho_k),\eta_k\tilde{s}_k\rangle]\\
        &=\E[\eta_k\langle{\nabla} f_k,{\rm P}_{\widetilde V_k}\tilde{\nabla} f_k\rangle+\eta_k\langle\nabla\lambda_kc_k,{\rm P}_{\widetilde V_k}\tilde{\nabla} f_k\rangle+\eta_k\langle\nabla c_k(\lambda_k+\rho_k c_k),\Delta P_k\rangle\\
        &\quad+\eta_k\rho_k w\|\nabla c_k c_k\|^2+ \eta_kw\langle\nabla\lambda_k\tilde\nabla c_k\tilde c_k,c_k\rangle+\eta_k\langle\nabla f_k+\nabla c_k\lambda_k+\rho_k\nabla c_k c_k,w\Delta c_k\rangle]\\
        &=\E[\eta_k\langle{\nabla} f_k,{\rm P}_{V_k}{\nabla} f_k\rangle+\eta_k\langle\nabla\lambda_k^\top{\rm P}_{\widetilde V_k}\tilde{\nabla} f_k+w\nabla\lambda_k\tilde\nabla c_k\tilde c_k,c_k\rangle+\eta_k\rho_k w\|\nabla c_k c_k\|^2\\
        &\quad+ \eta_k\langle\nabla f_k+\nabla c_k\lambda_k+\rho_k\nabla c_k c_k,\Delta P_k\rangle+\eta_k\langle\nabla f_k+\nabla c_k\lambda_k+\rho_k\nabla c_k c_k,w\Delta c_k\rangle]\\
        &\geq\E[\eta_k\|{\rm P}_{V_k}{\nabla} f_k\|^2]-L_\lambda\E[\eta_k(\|{\rm P}_{\widetilde V_k}\tilde{\nabla} f_k\|+w\|\tilde\nabla c_k\tilde c_k\|)\|c_k\|]+\E[\eta_k\rho_k w\|\nabla c_k c_k\|^2]\\
        &\quad-\E[\eta_k\|{\rm P}_{V_k}{\nabla} f_k\|(\|\Delta P_k\|+w\|\Delta c_k\|)]-L_c\E[\eta_k\rho_k\|c
        _k\|(\|\Delta P_k\|+w\|\Delta c_k\|)],
    \end{aligned}\ee
    where the first equality holds by the orthogonality of the two components of \(\tilde s_k\), and the inequality relies on the properties of the orthogonal projection $\mathrm{P}_{V_k}\nabla f_k$ combined with the bound $\|\nabla \lambda(x)\| \leq L_{\lambda}$ from Proposition \ref{le:lambda}.

    On the good event \(\mathcal{G}_k\), for the third item in R.H.S. of (\ref{L-times-s}), it is easy to attain from $\|a+b\|^2\geq\frac{1}{2}\|a\|^2-\|b\|^2$ and $\Delta J_k:=\tilde\nabla c_k-\nabla c_k$ that 
    \begin{align*}
        \E[\eta_k\rho_k w\|\nabla c_k c_k\|^2\mathbf{1}_{\mathcal{G}_k}]&=\E[\eta_k\rho_k w\|\nabla c_k c_k-\tilde\nabla c_k c_k+\tilde\nabla c_k c_k\|^2\mathbf{1}_{\mathcal{G}_k}]\\
        &\geq \frac{1}{2}\E[\eta_k\rho_k w\|\tilde\nabla c_k c_k\|^2\mathbf{1}_{\mathcal{G}_k}]- \E[\eta_k\rho_k w\| c_k \|^2\|\tilde\nabla c_k-\nabla c_k\|^2\mathbf{1}_{\mathcal{G}_k}]\\
        &\geq\frac{1}{2}\E[\eta_k\rho_k w\|\tilde\nabla c_k c_k\|^2\mathbf{1}_{\mathcal{G}_k}]-M\E[\eta_k\rho_k w\| c_k \|\|\Delta J_k\|^2\mathbf{1}_{\mathcal{G}_k}].
    \end{align*}
    Recalling the update rule for \(\rho_k\), we observe that on the good event \(\mathcal G_k\), the merit parameter satisfies $\rho_k\geq(4L_\lambda^2+2w^2L_c^2+2)/(w\mu_{\min}(\tilde\nabla c_k)^2)$, thereby implying that $\frac{1}{2}\rho_k w\|\tilde\nabla c_k c_k\|^2-2L_\lambda^2\|c_k\|^2\geq (w^2L_c^2+1)\|c_k\|^2$. It thus together with Young's inequality and the definition of $\tilde s_k$ indicates that
    \vspace{-3mm}
 \begin{align}
     &-\E[\langle\nabla\calL(x_k,\rho_k),\eta_k\tilde{s}_k\rangle \mathbf{1}_{\mathcal{G}_k}] \label{L-times-s-2} \\ 
     &\geq\E[\eta_k\|{\rm P}_{V_k}{\nabla} f_k\|^2\mathbf{1}_{\mathcal{G}_k}]+\frac{1}{2}\E[\eta_k\rho_k w\|\tilde\nabla c_k c_k\|^2\mathbf{1}_{\mathcal{G}_k}]-M\E[\eta_k\rho_k w\| c_k \|\|\Delta J_k\|^2\mathbf{1}_{\mathcal{G}_k}]\notag\\
        &\quad-L_\lambda\E\Big[\eta_k\Big(\frac{1}{4L_\lambda}\|{\rm P}_{\widetilde V_k}\tilde{\nabla} f_k\|^2+\frac{1}{4L_\lambda}w^2\|\tilde\nabla c_k\tilde c_k\|^2+2L_\lambda\|c_k\|^2\Big)\mathbf{1}_{\mathcal{G}_k}\Big]\notag\\
        &\quad-\E[\eta_k\big(\|{\rm P}_{V_k}{\nabla} f_k\|+L_c\rho_k\|c
        _k\|\big)(\|\Delta P_k\|+w\|\Delta c_k\|)\mathbf{1}_{\mathcal{G}_k}]\notag\\
        &=\E[\eta_k\|{\rm P}_{V_k}{\nabla} f_k\|^2\mathbf{1}_{\mathcal{G}_k}]+\E\big[\eta_k\big(\frac{1}{2}\rho_k w\|\tilde\nabla c_k c_k\|^2-2L_\lambda^2\|c_k\|^2\big)\mathbf{1}_{\mathcal{G}_k}\big]\notag\\
     &\quad-\frac{1}{4}\E[\eta_k\|\tilde s_k\|^2\mathbf{1}_{\mathcal{G}_k}]-M\E[\eta_k\rho_k w\| c_k \|\|\Delta J_k\|^2\mathbf{1}_{\mathcal{G}_k}]\notag\\
     &\quad-\E[\eta_k\big(\|{\rm P}_{V_k}{\nabla} f_k\|+L_c\rho_k\|c
        _k\|\big)(\|\Delta P_k\|+w\|\Delta c_k\|)\mathbf{1}_{\mathcal{G}_k}]\notag\\
     &\geq \E[\eta_k(\|{\rm P}_{V_k}{\nabla} f_k\|^2+(w^2L_c^2+1)\|c_k\|^2)\mathbf{1}_{\mathcal{G}_k}]-M\E[\eta_k\rho_k w\| c_k \|\|\Delta J_k\|^2\mathbf{1}_{\mathcal{G}_k}]\notag\\
     &\quad-\frac{1}{4}\E[\eta_k\|\tilde s_k\|^2\mathbf{1}_{\mathcal{G}_k}]-\E[\eta_k\big(\|{\rm P}_{V_k}{\nabla} f_k\|+L_c\rho_k\|c
        _k\|\big)(\|\Delta P_k\|+w\|\Delta c_k\|)\mathbf{1}_{\mathcal{G}_k}].\notag
 \end{align}
  On the other hand, on the complementary event \(\mathcal G_k^c\), the merit parameter satisfies $\rho_k\geq(2L_\lambda^2+1+w^2L_c^2)/(w\nu^2)$. Then it holds from (\ref{L-times-s}) that
\begin{align*}
 &-\E[\langle\nabla\calL(x_k,\rho_k),\eta_k\tilde{s}_k\rangle \mathbf{1}_{\mathcal{G}_k^c}]\\
     &\geq\E[\eta_k\|{\rm P}_{V_k}{\nabla} f_k\|^2\mathbf{1}_{\mathcal{G}_k^c}]+\E[\eta_k\rho_k w\|\nabla c_k c_k\|^2\mathbf{1}_{\mathcal{G}_k^c}]\notag\\
     &\quad-L_\lambda\E\Big[\eta_k\Big(\frac{1}{4L_\lambda}\|{\rm P}_{\widetilde V_k}\tilde{\nabla} f_k\|^2+\frac{1}{4L_\lambda}w^2\|\tilde\nabla c_k\tilde c_k\|^2+2L_\lambda\|c_k\|^2\Big)\mathbf{1}_{\mathcal{G}_k^c}\Big]\\
        &\quad-\E[\eta_k\big(\|{\rm P}_{V_k}{\nabla} f_k\|+L_c\rho_k\|c
        _k\|\big)(\|\Delta P_k\|+w\|\Delta c_k\|)\mathbf{1}_{\mathcal{G}_k^c}]\\        
        &\geq\E[\eta_k\|{\rm P}_{V_k}{\nabla} f_k\|^2\mathbf{1}_{\mathcal{G}_k^c}]+\E[\eta_k(\rho_k w\nu^2-2L_\lambda^2)\|c_k\|^2\mathbf{1}_{\mathcal{G}_k^c}]-\frac{1}{4}\E[\eta_k\|\tilde s_k\|^2\mathbf{1}_{\mathcal{G}_k^c}]\\
        &\quad-\E[\eta_k\big(\|{\rm P}_{V_k}{\nabla} f_k\|+L_c\rho_k\|c
        _k\|\big)(\|\Delta P_k\|+w\|\Delta c_k\|)\mathbf{1}_{\mathcal{G}_k^c}]\\     
     &\geq \E[\eta_k(\|{\rm P}_{V_k}{\nabla} f_k\|^2+(w^2L_c^2+1)\|c_k\|^2)\mathbf{1}_{\mathcal{G}_k^c}]-\frac{1}{4}\E[\eta_k\|\tilde s_k\|^2\mathbf{1}_{\mathcal{G}_k^c}]\\
        &\quad-\E[\eta_k\big(\|{\rm P}_{V_k}{\nabla} f_k\|+L_c\rho_k\|c
        _k\|\big)(\|\Delta P_k\|+w\|\Delta c_k\|)\mathbf{1}_{\mathcal{G}_k^c}],
 \end{align*}
 where the second inequality comes from Assumption \ref{a4}. Combining the bounds obtained on $\mathcal{G}_k$ and $\mathcal{G}_k^c$, we conclude that
\be\begin{aligned}\label{L-times-s-uni}
    & -\E[\langle\nabla\calL(x_k,\rho_k),\eta_k\tilde{s}_k\rangle]=-\E[\langle\nabla\calL(x_k,\rho_k),\eta_k\tilde{s}_k\rangle \mathbf{1}_{\mathcal G_k}]
     -\E[\langle\nabla\calL(x_k,\rho_k),\eta_k\tilde{s}_k\rangle \mathbf{1}_{\mathcal G_k^c}]\\
     &\geq \E[\eta_k(\|{\rm P}_{V_k}{\nabla} f_k\|^2+(w^2L_c^2+1)\|c_k\|^2)]-\frac{1}{4}\E[\eta_k\|\tilde s_k\|^2]-wM\E[\eta_k\rho_k \| c_k \|\|\Delta J_k\|^2]\\
     &\quad-\E\big[\big(\eta_k\|{\rm P}_{V_k}{\nabla} f_k\|+L_c\eta_k\rho_k\|c
        _k\|\big)(\|\Delta P_k\|+w\|\Delta c_k\|)\big].
\end{aligned}\ee
Substituting (\ref{L-times-s-uni}) into (\ref{L-expand}) yields the conclusion.
\end{proof}
The next lemma establishes an average bound on the stationarity and feasibility residuals along the sequence generated by Algorithm \ref{alg1}.
\begin{lemma}
     Under Assumptions {\ref{a1}-\ref{a4} and \ref{a5}}, it holds that
      \begin{align}\label{eq:s2}
       & \frac{1}{K} \sum_{k=1}^{K} \E\left[\|{\rm P}_{V_k} \nabla f_k\|^2 + \|c_k\|^2 \right]\\&\le \frac{4\E[\calL(x_1,\rho_1)-\calL(x_{K+1},\rho_{K+1})]}{\eta_{\min} K}+\frac{2M^2\rho_{\max}}{\eta_{\min}K}+\frac{\bar\eta C_s^2}{K}\sum_{k=1}^{K} (4\eta_{\max}L_{\max}+2)\tilde\sigma_k^2\notag\\
        &\quad + 4\bar\eta\tilde C\max\{1,\rho_{\max} L_c\}\frac{1}{K} \sum_{k=1}^{K} \tilde\sigma_k\left((\E[\|{\rm P}_{V_k} \nabla f_k\|^2])^{1/2}+(\E[\|c_k\|^2])^{1/2}\right),\notag
    \end{align}
    where $\tilde C:=(C_p^2+16G^2/\nu^2)^{1/2}+w(2M^2+2L_c^2)^{1/2}+2wM$ and $\bar\eta :=\eta_{\max}/\eta_{\min}$. 
\end{lemma}
\begin{proof}
First, recall from \eqref{eq:descent-L} the error terms involving \(\Delta P_k\), \(\Delta c_k\), and \(\Delta J_k\). It follows from Cauchy--Schwarz inequality that
\be\begin{aligned}\label{eta_Pv}
&\mathbb E[\eta_k\|{\rm P}_{V_k}\nabla f_k\|(\|\Delta P_k\|+w\|\Delta c_k\|)] \\
&\le \big(\mathbb E[\eta_k^2\|{\rm P}_{V_k}\nabla f_k\|^2]\big)^{1/2} \big(\mathbb E[\|\Delta P_k\|^2]\big)^{1/2} + w\big(\mathbb E[\eta_k^2\|{\rm P}_{V_k}\nabla f_k\|^2]\big)^{1/2}\big(\mathbb E[\|\Delta c_k\|^2]\big)^{1/2} \\
&= \big(\mathbb E[\eta_k^2\|{\rm P}_{V_k}\nabla f_k\|^2]\big)^{1/2} \Big(\big(\mathbb E[\|\Delta P_k\|^2]\big)^{1/2}+w\big(\mathbb E[\|\Delta c_k\|^2]\big)^{1/2}\Big).
\end{aligned}\ee
The tower property, together with (\ref{pv-pv}) and the {estimate on $\mathcal{G}_k^c$ in} (\ref{Pv-bad}) yields \[\mathbb E[\|\Delta P_k\|^2]
=\mathbb E\!\left[\bar{\mathbb E}_k[\|\Delta P_k\|^2\mathbf 1_{\mathcal G_k}]+\bar{\mathbb E}_k[\|\Delta P_k\|^2\mathbf 1_{\mathcal G_k^c}]\right]\le \big(C_p^2+\frac{16G^2}{\nu^2}\big) \tilde\sigma_k^2:=\widehat{C}_p^2\tilde\sigma_k^2.\]
Moreover, it follows from Assumption \ref{a5} that, for the term $\|\Delta c_k\|$,
\begin{align*}
    \E[\|\Delta c_k\|^2]&=\E[\bar{\E}_k[\|\Delta c_k\|^2]]\leq2\E\left[\bar{\E}_k[\|\tilde\nabla c_k-\nabla c_k\|^2\|\tilde c_k\|^2]+\bar{\E}_k[\|\nabla c_k\|^2\|\tilde c_k-c_k\|^2]\right]\\
    &\leq 2(M^2+L_c^2)\tilde\sigma_k^2:=C_v^2\tilde\sigma_k^2.
\end{align*}
Plugging the above estimates into (\ref{eta_Pv}), we arrive at
\begin{align}\label{eq:eta-Pv}
\mathbb E[\eta_k\|{\rm P}_{V_k}\nabla f_k\|(\|\Delta P_k\|+w\|\Delta c_k\|)]
\le (\widehat{C}_p+wC_v)\tilde\sigma_k\big(\mathbb E[\eta_k^2\|{\rm P}_{V_k}\nabla f_k\|^2]\big)^{1/2}.
\end{align}
An analogous argument for the other term involving 
\((\|\Delta P_k\|+w\|\Delta c_k\|)\) gives
\be\begin{aligned}\label{eq:eta-c1}
&L_c\,\mathbb E[\eta_k\rho_k\|c_k\|(\|\Delta P_k\|+w\|\Delta c_k\|)]\\
&\le\rho_{\max}L_c\big(\mathbb E[\eta_k^2\|c_k\|^2]\big)^{1/2}
\Big((\mathbb E[\|\Delta P_k\|^2])^{1/2}
+w(\mathbb E[\|\Delta c_k\|^2])^{1/2}
\Big) \\
&\le \rho_{\max}L_c(\widehat{C}_p+wC_v)\tilde\sigma_k\big(\mathbb E[\eta_k^2\|c_k\|^2]\big)^{1/2}.
\end{aligned}\ee
In view of the bound $\|\Delta J_k\|^2 \leq 4L_c^2$ and  $\E[\|\Delta J_k\|^2]=\E[\bar{\E}_k\|\Delta J_k\|^2] \leq \tilde\sigma_k^2$ provided by Assumption \ref{a5}, the term involving $\Delta J_k$ in \eqref{eq:descent-L} satisfies
\be\begin{aligned}\label{eq:eta-c2}
    &wM\E[\eta_k\rho_k \| c_k \|\|\Delta J_k\|^2]\\
    &\leq wM\rho_{\max}\big(\mathbb E[\eta_k^2\|c_k\|^2]\big)^{1/2}\big(\mathbb E[\|\Delta J_k\|^4]\big)^{1/2}\leq2wM\rho_{\max}L_c\tilde\sigma_k\big(\mathbb E[\eta_k^2\|c_k\|^2]\big)^{1/2}.
\end{aligned}\ee
For notational convenience, we define $\tilde C := \widehat{C}_p + wC_v+2wM$. Incorporating (\ref{eq:eta-Pv})-(\ref{eq:eta-c2}) into (\ref{eq:descent-L}) results in
%\begin{align}\label{L-times-s-3}
%-\mathbb E[\langle \nabla \calL(x_k,\rho_k),\eta_k\tilde s_k\rangle]&\ge\mathbb E\big[\eta_k(\|{\rm P}_{V_k}\nabla f_k\|^2+(w^2L_c^2+1)\|c_k\|^2)-\frac{\eta_k}{4}\|\tilde s_k\|^2\big]\notag\\
%&\quad-\tilde C \tilde\sigma_k\big(\mathbb E[\eta_k^2\|{\rm P}_{V_k}\nabla f_k\|^2]\big)^{1/2}-\rho_{\max}L_c\tilde C \tilde\sigma_k\big(\mathbb E[\eta_k^2\|c_k\|^2]\big)^{1/2}.
%\end{align}
%Then substituting \eqref{L-times-s-3} into \eqref{L-expand} yields
    \begin{align*}
     &\E[\calL(x_{k+1},\rho_k)]\\
      & \le\E[\calL(x_k,\rho_k)] - \E\left[\eta_k\left( \|{\rm P}_{V_k} \nabla f_k\|^2 + (w^2L_c^2+1)\|c_k\|^2 \right)\right]+\tilde C \tilde\sigma_k\big(\mathbb E[\eta_k^2\|{\rm P}_{V_k}\nabla f_k\|^2]\big)^{\frac12} \\
        & \quad+\rho_{\max}L_c\tilde C \tilde\sigma_k\big(\mathbb E[\eta_k^2\|c_k\|^2]\big)^{\frac12}+ \frac{1}{4}\E[(2\eta_k^2L_k+\eta_k) \|\tilde{s}_k\|^2]\\
        & \le\E[\calL(x_k,\rho_k)] - \E[\eta_k\left( \|{\rm P}_{V_k} \nabla f_k\|^2 + (w^2L_c^2+1)\|c_k\|^2 \right)]+\tilde C \tilde\sigma_k\big(\mathbb E[\eta_k^2\|{\rm P}_{V_k}\nabla f_k\|^2]\big)^{\frac12} \\
        & \quad+\rho_{\max}L_c\tilde C \tilde\sigma_k\big(\mathbb E[\eta_k^2\|c_k\|^2]\big)^{\frac12}+ \E\left[\left(\eta_k^2L_k+\frac{\eta_k}{2}\right)(\|{s}_k\|^2+\|\tilde{s}_k-{s}_k\|^2)\right]\\
        & \le\E[\calL(x_k,\rho_k)] - \E[\eta_k\left( \|{\rm P}_{V_k} \nabla f_k\|^2 + (w^2L_c^2+1)\|c_k\|^2 \right)]+\tilde C \tilde\sigma_k\big(\mathbb E[\eta_k^2\|{\rm P}_{V_k}\nabla f_k\|^2]\big)^{\frac12}\\
        &\quad+\rho_{\max}L_c\tilde C \tilde\sigma_k\big(\mathbb E[\eta_k^2\|c_k\|^2]\big)^{\frac12}+\E\left[\left(\eta_k^2L_k+\frac{\eta_k}{2}\right)\|\tilde{s}_k-{s}_k\|^2\right] \\
        & \quad+ \E\left[(\eta_k^2L_k+\frac{\eta_k}{2})\left( \|{\rm P}_{V_k} \nabla f_k\|^2 + w^2 L_c^2 \|c_k\|^2 \right)\right]\\
        & =\E[\calL(x_k,\rho_k)] - \E\left[\frac{\eta_k}{2} (1 - 2\eta_k L_k ) \|{\rm P}_{V_k} \nabla f_k\|^2+\left(\frac{\eta_k}{2} w^2 L_c^2 (1 -2\eta_k L_k)+\eta_k\right) \|c_k\|^2\right]\\
        & \quad+\tilde C \tilde\sigma_k\left[\big(\mathbb E[\eta_k^2\|{\rm P}_{V_k}\nabla f_k\|^2]\big)^{\frac12}+\rho_{\max}L_c\big(\mathbb E[\eta_k^2\|c_k\|^2]\big)^{\frac12} \right]+(\eta_{\max}^2L_{\max}+\frac{\eta_{\max}}{2})C_s^2 \tilde\sigma_k^2\\
       & \le \E[\calL(x_k,\rho_k)] - \E\left[\frac{\eta_k}{4} \left( \|{\rm P}_{V_k} \nabla f_k\|^2 + \|c_k\|^2 \right)\right]+(\eta_{\max}^2L_{\max}+\frac{\eta_{\max}}{2}) C_s^2\tilde\sigma_k^2\\
       &\quad+\tilde C \tilde\sigma_k\big(\mathbb E[\eta_k^2\|{\rm P}_{V_k}\nabla f_k\|^2]\big)^{\frac12}+\rho_{\max}L_c\tilde C \tilde\sigma_k\big(\mathbb E[\eta_k^2\|c_k\|^2]\big)^{\frac12},
    \end{align*}
    where the last inequality follows from $1-2\eta_kL_k\geq\frac{1}{2}$ given by (\ref{eq:eta-update}). Note that
    \begin{align}\label{Lk+1-Lk}
        \calL(x_{k+1},\rho_{k+1})-\calL(x_{k},\rho_{k})
        &=\calL(x_{k+1},\rho_{k+1})-\calL(x_{k+1},\rho_{k})+\calL(x_{k+1},\rho_{k})-\calL(x_{k},\rho_{k})\notag\\
        &=\frac{\rho_{k+1}-\rho_k}{2}\|c(x_{k+1})\|^2+\calL(x_{k+1},\rho_{k})-\calL(x_{k},\rho_{k}).
    \end{align}
    Summing the above inequality over \(k = 1,\ldots, K\) and rearranging the terms leads to \eqref{eq:s2}. 
    % \begin{align*}
    %     &\frac{1}{K} \sum_{k=1}^{K} \E[\|{\rm P}_{V_k} \nabla f_k\|^2 + \|c_k\|^2]  \\
    %     & \le \frac{4\E[\calL(x_1,\rho_1)-\calL(x_{K+1},\rho_{K+1})]}{\eta_{\min} K}+\frac{2M^2\rho_{\max}}{\eta_{\min}K}+\frac{\bar\eta C_s^2}{K}\sum_{k=1}^{K}\E[(4\eta_kL_k+2)\tilde\sigma_k^2]\\
    %     &\quad + 4\bar\eta\tilde C\max\{1,\rho_{\max} L_c\}\frac{1}{K} \sum_{k=1}^{K} \tilde\sigma_k\left(\E[\|{\rm P}_{V_k}\nabla f_k\|^2]\big)^{1/2}+\E[\|c_k\|^2]\big)^{1/2}\right),
    % \end{align*}
    % where $\bar\eta$ is defined as $\eta_{\max}/\eta_{\min}$. 
    The proof is thus completed.
\end{proof}
%\begin{remark}
%Note that, by the monotone update rule for \(\rho_k\), once the bad event occurs at some iteration, \(\rho_k\) is raised to the conservative level determined by the safeguard \(\nu^2/4\) and will not be updated afterward. Consequently, even if subsequent iterations fall into the good event again, \(\rho_k\) remains at least as large as the conservative value selected when the bad event occurs. In this case, one may continue the analysis by using the same conservative lower bound as in the bad-event case. Therefore, the unified estimate used later remains valid. Moreover, we will subsequently show that the probability of the bad event is sufficiently small.
%\end{remark}

We next derive an upper bound on $\sum_{k=1}^{K}\E[\|\tilde{s}_k\|^2]$, which will be used later to support the oracle complexity analysis.
\begin{lemma}\label{le:sum-sk}
Under Assumptions {\ref{a1}-\ref{a4} and \ref{a5}}, it holds that 
    \begin{align}\label{eq:sum-sk}
        \sum_{k=1}^{K}\E[\|\tilde{s}_k\|^2]&\leq \frac{8\E[\calL(x_1,\rho_1)-\calL(x_{K+1},\rho_{K+1})]}{\eta_{\min} }+8\bar\eta C_s^2 \sum_{k=1}^{K} \tilde\sigma_k^2+\frac{4M^2\rho_{\max}}{\eta_{\min}}\notag\\
        &\quad+ 8\bar\eta\tilde C\max\{1,\rho_{\max}L_c\}  \sum_{k=1}^{K} \tilde\sigma_k\left((\E[\|{\rm P}_{V_k} \nabla f_k\|^2])^{1/2}+(\E[\|c_k\|^2])^{1/2}\right).
    \end{align}
\end{lemma}
\begin{proof}
    It follows from (\ref{L-times-s-2}) that, on the good event $\mathcal{G}_k$,
    \begin{align*}
     &-\E[\langle\nabla\calL(x_k,\rho_k),\eta_k\tilde{s}_k\rangle\mathbf{1}_{\mathcal{G}_k}]\\
     &\geq \E\Big[\Big(\eta_k\|{\rm P}_{V_k}{\nabla} f_k\|^2+(\frac{1}{2}\rho_k w\|\tilde\nabla c_k c_k\|^2-2L_\lambda^2\|c_k\|^2)\eta_k-\frac{\eta_k}{4}\|\tilde s_k\|^2\Big)\mathbf{1}_{\mathcal{G}_k}\Big]\\
        &\quad-M\E[\eta_k\rho_k w\| c_k \|\|\Delta J_k\|^2\mathbf{1}_{\mathcal{G}_k}]-\E[\eta_k\|{\rm P}_{V_k}{\nabla} f_k\|(\|\Delta P_k\|+w\|\Delta c_k\|)\mathbf{1}_{\mathcal{G}_k}]\\
        &\quad-L_c\E[\eta_k\rho_k\|c
        _k\|(\|\Delta P_k\|+w\|\Delta c_k\|)\mathbf{1}_{\mathcal{G}_k}]\\
     &\geq \E\big[\big(\eta_k(\|{\rm P}_{V_k}{\nabla} f_k\|^2+w^2L_c^2\|c_k\|^2)-\frac{\eta_k}{4}\|\tilde s_k\|^2\big)\mathbf{1}_{\mathcal{G}_k}\big]-M\E[\eta_k\rho_k w\| c_k \|\|\Delta J_k\|^2\mathbf{1}_{\mathcal{G}_k}]\\
        &\quad-\E\big[\big(\eta_k\|{\rm P}_{V_k}{\nabla} f_k\|+L_c\eta_k\rho_k\|c
        _k\|\big)(\|\Delta P_k\|+w\|\Delta c_k\|)\mathbf{1}_{\mathcal{G}_k}\big]\\
        &\geq\E[\eta_k\|s_k\|^2\mathbf{1}_{\mathcal{G}_k}]-\frac{1}{4}\E[\eta_k\|\tilde s_k\|^2\mathbf{1}_{\mathcal{G}_k}]-M\E[\eta_k\rho_k w\| c_k \|\|\Delta J_k\|^2\mathbf{1}_{\mathcal{G}_k}]\\
        &\quad-\E\big[\big(\eta_k\|{\rm P}_{V_k}{\nabla} f_k\|+L_c\eta_k\rho_k\|c
        _k\|\big)(\|\Delta P_k\|+w\|\Delta c_k\|)\mathbf{1}_{\mathcal{G}_k}\big]\\
        &\geq \frac{1}{4}\E[\eta_k\|\tilde s_k\|^2\mathbf{1}_{\mathcal{G}_k}]-\E[\eta_k\|s_k-\tilde s_k\|^2\mathbf{1}_{\mathcal{G}_k}]-M\E[\eta_k\rho_k w\| c_k \|\|\Delta J_k\|^2\mathbf{1}_{\mathcal{G}_k}]\\
        &\quad-\E\big[\big(\eta_k\|{\rm P}_{V_k}{\nabla} f_k\|+L_c\eta_k\rho_k\|c
        _k\|\big)(\|\Delta P_k\|+w\|\Delta c_k\|)\mathbf{1}_{\mathcal{G}_k}\big],
    \end{align*}
     where the second inequality holds due to the condition $\frac{1}{2}\rho_k w\|\tilde\nabla c_k c_k\|^2-2L_\lambda^2\|c_k\|^2\geq w^2L_c^2\|c_k\|^2$ guaranteed by the update rule of $\rho_k$, and the last inequality follows from $\|a+b\|^2\geq\frac{1}{2}\|a\|^2-\|b\|^2$. Similarly, on the bad event $\mathcal{G}_k^c$, it holds from $\rho_k\geq(2L_\lambda^2+w^2L_c^2)/(w\nu^2)$ that
     \begin{align*}
         &-\E[\langle\nabla\calL(x_k,\rho_k),\eta_k\tilde{s}_k\rangle\mathbf{1}_{\mathcal{G}_k^c}]\\
     &\geq \E\big[\big(\eta_k\|{\rm P}_{V_k}{\nabla} f_k\|^2+(\rho_k w\nu^2-2L_\lambda^2)\eta_k\|c_k\|^2-\frac{\eta_k}{4}\|\tilde s_k\|^2\big)\mathbf{1}_{\mathcal{G}_k^c}\big]\\
        &\quad-\E\big[\big(\eta_k\|{\rm P}_{V_k}{\nabla} f_k\|+L_c\eta_k\rho_k\|c
        _k\|\big)(\|\Delta P_k\|+w\|\Delta c_k\|)\mathbf{1}_{\mathcal{G}_k^c}\big]\\
     &\geq \E[\eta_k\|s_k\|^2\mathbf{1}_{\mathcal{G}_k^c}]-\frac{1}{4}\E[\eta_k\|\tilde s_k\|^2\mathbf{1}_{\mathcal{G}_k^c}]-\E[\eta_k\|{\rm P}_{V_k}{\nabla} f_k\|(\|\Delta P_k\|+w\|\Delta c_k\|)\mathbf{1}_{\mathcal{G}_k^c}]\\
        &\quad-L_c\E[\eta_k\rho_k\|c
        _k\|(\|\Delta P_k\|+w\|\Delta c_k\|)\mathbf{1}_{\mathcal{G}_k^c}]\\
        &\geq \frac{1}{4}\E[\eta_k\|\tilde s_k\|^2\mathbf{1}_{\mathcal{G}_k^c}]-\E[\eta_k\|s_k-\tilde s_k\|^2\mathbf{1}_{\mathcal{G}_k^c}]-\E[\eta_k\|{\rm P}_{V_k}{\nabla} f_k\|(\|\Delta P_k\|+w\|\Delta c_k\|)\mathbf{1}_{\mathcal{G}_k^c}]\\
        &\quad-L_c\E[\eta_k\rho_k\|c
        _k\|(\|\Delta P_k\|+w\|\Delta c_k\|)\mathbf{1}_{\mathcal{G}_k^c}].
     \end{align*}
    Combining the bounds obtained on $\mathcal{G}_k$ and $\mathcal{G}_k^c$, respectively, we obtain
    \begin{align*}
          \E[\calL(x_{k+1},\rho_k)]
          &\leq\E[\calL(x_k,\rho_k)]+\E[\langle\nabla\calL(x_k,\rho_k),\eta_k\tilde{s}_k\rangle]+\frac{1}{2}\E[\eta_k^2L_k\|\tilde{s}_k\|^2]\\
          &\leq \E[\calL(x_k,\rho_k)]-\frac{1}{4}\E[(\eta_k-2\eta_k^2 L_k)\|\tilde{s}_k\|^2]+\E[\eta_k\|s_k-\tilde s_k\|^2]\\
          &\quad+\tilde C \tilde\sigma_k\big(\mathbb E[\eta_k^2\|{\rm P}_{V_k}\nabla f_k\|^2]\big)^{1/2}+\rho_{\max}L_c\tilde C \tilde\sigma_k\big(\mathbb E[\eta_k^2\|c_k\|^2]\big)^{1/2}\\
          &\leq\E[\calL(x_k,\rho_k)]-\frac{1}{8}\E[\eta_k\|\tilde{s}_k\|^2]+\E[\eta_k\|s_k-\tilde s_k\|^2]\\
          &\quad+\tilde C\max\{1,\rho_{\max} L_c\}\tilde\sigma_k\left(\E[\eta_k^2\|{\rm P}_{V_k}\nabla f_k\|^2]\big)^{1/2}+\E[\eta_k^2\|c_k\|^2]\big)^{1/2}\right),
    \end{align*}
    where the second inequality follows from (\ref{eq:eta-Pv}) and (\ref{eq:eta-c1}), and the last is due to $\eta_k\leq\frac{1}{4L_k}$. Summing it from $k=1$ to $K$ and using (\ref{Lk+1-Lk}), we derive
    \begin{align*}
        \sum_{k=1}^{K}\E[\|\tilde{s}_k\|^2]&\leq \frac{8\E[\calL(x_1,\rho_1)-\calL(x_{K+1},\rho_{K+1})]}{\eta_{\min} }+8\bar\eta C_s^2 \sum_{k=1}^{K} \tilde\sigma_k^2+\frac{4M^2\rho_{\max}}{\eta_{\min}}\\
        &\quad+ 8\bar\eta\tilde C\max\{1,\rho_{\max} L_c\}\sum_{k=1}^{K}\tilde\sigma_k\left(\E[\|{\rm P}_{V_k}\nabla f_k\|^2]\big)^{1/2}+\E[\|c_k\|^2]\big)^{1/2}\right).
    \end{align*}
    Hence, we obtain the conclusion.
\end{proof}

\section{Complexity analysis}
In this section, we establish the oracle complexity of Algorithm \ref{alg1} for finding a stochastic \(\epsilon\)-KKT point of problem (\ref{p}). The analysis starts with expectation bounds on the estimation errors generated by the truncated SPIDER-type estimators.

\begin{lemma}\label{lem:spider-est}
Suppose that Assumptions \ref{a1}-\ref{a2} holds. Then for any $k\ge 1$  
%with $q(k)=\lfloor (k-1)/\tau\rfloor\tau+1$ denoting the starting point of the epoch containing \(k\), 
it holds that
\begin{align}
\bar{\mathbb E}_k\!\left[\|\tilde \nabla f_k-\nabla f_k\|^2\right]
&\le
\frac{\sigma_f^2}{B_{q(k)}^f}
+
\sum_{p=q(k)+1}^{k}\bar{\mathbb E}_k\!\bigg[\frac{L_f^2}{B_p^f}\|x_p-x_{p-1}\|^2\bigg], \label{eq:spider-f}\\
\bar{\mathbb E}_k\!\left[\|\tilde \nabla c_k-\nabla c_k\|^2\right]
&\le
\frac{m\sigma_J^2}{B_{q(k)}^J}
+
\sum_{p=q(k)+1}^{k}\bar{\mathbb E}_k\!\left[
\frac{mL_J^2}{B_p^J}
\|x_p-x_{p-1}\|^2\right], \label{eq:spider-J}\\
\bar{\mathbb E}_k\!\left[\|\tilde c_k-c_k\|^2\right]
&\le
\frac{\sigma_c^2}{B_{q(k)}^c}
+
\sum_{p=q(k)+1}^{k}\bar{\mathbb E}_k\!\left[
\frac{L_c^2}{B_p^c}
\|x_p-x_{p-1}\|^2\right]. \label{eq:spider-c}
\end{align}
\end{lemma}

\begin{proof}
For any fixed $k\ge 1$, we first consider the case $k=q(k)$. Since $\nabla f_k\in \mathbb B(G)$ by Assumption~\ref{a1}, the nonexpansiveness of the projection operator implies
\begin{align*}
\|\tilde \nabla f_k-\nabla f_k\|^2
&=\Big\|{\rm P}_{\mathbb B(G)}
\big(\frac{1}{B_k^f}\sum_{\xi\in \mathcal{B}_k^f}\nabla F(x_k,\xi)\big)
-{\rm P}_{\mathbb B(G)}(\nabla f_k)\Big\|^2 \\
&\le
\Big\|
\frac{1}{B_k^f}\sum_{\xi\in \mathcal{B}_k^f}
\big(\nabla F(x_k,\xi)-\nabla f_k\big)
\Big\|^2.
\end{align*}
Taking conditional expectation with respect to $\bar{\mathcal F}_k=\mathcal F_k$ and using Assumption~\ref{a2}, we obtain
\begin{equation}
\bar{\mathbb E}_k[\|\tilde \nabla f_k-\nabla f_k\|^2]
\le\frac{\sigma_f^2}{B_k^f}.
\label{eq:spider-refresh-f-bound}
\end{equation}
Next, let $p\in\{{q(k)}+1,\ldots,k\}$. It follows from \eqref{eq:spider-k} and the nonexpansiveness of the projection operator that
\begin{align*}
&\|\tilde \nabla f_p-\nabla f_p\|^2\leq\Big\|\tilde \nabla f_{p-1}+\frac{1}{B_p^f}\sum_{\xi\in \mathcal{B}_p^f}\big(\nabla F(x_p,\xi)-\nabla F(x_{p-1},\xi)\big)
-\nabla f_p\Big\|^2 \\
&=\Big\|\tilde \nabla f_{p-1}-\nabla f_{p-1}+\frac{1}{B_p^f}\sum_{\xi\in \mathcal{B}_p^f}\big[\nabla F(x_p,\xi)-\nabla F(x_{p-1},\xi)-(\nabla f_p-\nabla f_{p-1})\big]\Big\|^2.
\end{align*}
For brevity, define $\delta_p^f(\xi)
:=
\nabla F(x_p,\xi)-\nabla F(x_{p-1},\xi)
-(\nabla f_p-\nabla f_{p-1})$. Taking conditional expectation with respect to \(\mathcal F_p\) in the above inequality, and using the fact that $\tilde \nabla f_{p-1}$, $\nabla f_{p-1}$, $x_p$, and $x_{p-1}$ are $\mathcal F_p$-measurable, it holds that
\begin{align*}
&\mathbb E_p[\|\tilde \nabla f_p-\nabla f_p\|^2]=\|\tilde \nabla f_{p-1}-\nabla f_{p-1}\|^2+\mathbb E_p\Big[
\big\|
\frac{1}{B_p^f}\sum_{\xi\in \mathcal{B}_p^f}
\delta_p^f(\xi)\big\|^2\Big]\\
&\leq\|\tilde \nabla f_{p-1}-\nabla f_{p-1}\|^2+\frac{1}{(B_p^f)^2}\sum_{\xi\in \mathcal{B}_p^f}
\mathbb E_p[\left\|\nabla F(x_p,\xi)-\nabla F(x_{p-1},\xi)\right\|^2]\\
&\leq \|\tilde \nabla f_{p-1}-\nabla f_{p-1}\|^2+\frac{L_f^2}{B_p^f}\|x_p-x_{p-1}\|^2, 
\end{align*}
where the equality uses the relations $\mathbb E_p\!\left[
\delta_p^f(\xi)\right]=0$,
the first inequality comes from
    \begin{align*}
		2\E_p[\langle\frac{1}{B_p^f} \sum_{\xi \in \mathcal{B}_p^f} \left(\nabla F (x_p;\xi) - \nabla F (x_{p-1};\xi)\right),\nabla{f}_p-\nabla{f}_{p-1}\rangle]&=2\|\nabla{f}_p-\nabla{f}_{p-1}\|^2,
    \end{align*}
and the last inequality is due to Assumption \ref{a2}. Then taking the conditional expectation $\bar{\mathbb E}_k[\cdot]$ on both sides of the preceding inequality, and invoking the tower property together with the inclusion $\bar{\mathcal F}_k = \mathcal F_{q(k)} \subseteq \mathcal F_p$, we deduce that
\[
\bar{\mathbb E}_k\!\big[\|\tilde \nabla f_p-\nabla f_p\|^2\big]
\le
\bar{\mathbb E}_k\!\big[\|\tilde \nabla f_{p-1}-\nabla f_{p-1}\|^2\big]
+\bar{\mathbb E}_k\!\bigg[\frac{L_f^2}{B_p^f}\|x_p-x_{p-1}\|^2\bigg].
\]
Iterating the above recursion from $p=q+1$ to $p=k$ and using \eqref{eq:spider-refresh-f-bound} gives
\begin{align*}
\bar{\mathbb E}_k\big[\|\tilde \nabla f_k-\nabla f_k\|^2\big]
&\le
\bar{\mathbb E}_k\big[\|\tilde \nabla f_q-\nabla f_q\|^2\big]
+\sum_{p=q+1}^{k}\bar{\mathbb E}_k\!\bigg[\frac{L_f^2}{B_p^f}\|x_p-x_{p-1}\|^2\bigg] \\
&\le\frac{\sigma_f^2}{B_q^f}+\sum_{p=q+1}^{k}
\bar{\mathbb E}_k\!\bigg[\frac{L_f^2}{B_p^f}\|x_p-x_{p-1}\|^2\bigg].
\end{align*}
Similarly, the upper bounds for stochastic variance of $\tilde\nabla c_k$ and $\tilde c_k$ can be established as in (\ref{eq:spider-J}) and (\ref{eq:spider-c}), respectively.
\end{proof}

Building upon Lemma \ref{lem:spider-est}, we derive the following corollary, thereby confirming the validity of Assumption \ref{a5}.
\begin{corollary}\label{lem:estimator-error}
Under the setting of Lemma \ref{lem:spider-est}, and for any given constant $\bar\epsilon \in (0,1)$, suppose that the sequence $\{B_k\}$ satisfies
\begin{align}\label{eq:batch-b0}
    B_k^f=\left\lceil\frac{2\sigma_f^2}{\bar\epsilon^2}\right\rceil,\quad B_k^J=\left\lceil\frac{2m\sigma_J^2}{\bar\epsilon^2}\right\rceil,\quad B_k^c=\left\lceil\frac{2\sigma_c^2}{\bar\epsilon^2}\right\rceil,\quad\text{for $k=q(k)$},
\end{align}
and 
\be\begin{aligned}\label{eq:batch-bk}
B_k^f&=\max\left\{1,\left\lceil\frac{2\tau L_f^2\|x_k-x_{k-1}\|^2}{\bar\epsilon^2}\right\rceil\right\},\,
B_k^J=\max\left\{1,\left\lceil\frac{2\tau mL_J^2\|x_k-x_{k-1}\|^2}{\bar\epsilon^2}\right\rceil\right\},\\
B_k^c&=\max\left\{1,\left\lceil\frac{2\tau L_c^2\|x_k-x_{k-1}\|^2}{\bar\epsilon^2}\right\rceil\right\},\quad \mbox{for $k> q(k)$,}
\end{aligned}\ee
then for all \(k\geq 1\),
\begin{align}
\bar{\mathbb E}_k[\|\tilde \nabla f_k-\nabla f_k\|^2]\leq \bar\epsilon^2,\quad
\bar{\mathbb E}_k[\|\tilde \nabla c_k-\nabla c_k\|^2]\leq \bar\epsilon^2,\quad
\bar{\mathbb E}_k[\|\tilde c_k-c_k\|^2]\leq \bar\epsilon^2. \label{eq:uniform-rec}
\end{align}
\end{corollary}

Using the simplified notation $C_m := \bar\eta\tilde C\max\{1, \rho_{\max} L_c\}$, we now proceed to analyze the oracle complexity under the parameter choices given by
\be\label{varepsilon-rm} \begin{aligned}
    &\bar\epsilon=\frac{\epsilon}{8(2C_{m}+(\eta_{\max}L_{
    \max}+1)^{\frac{1}{2}}{\bar\eta}^{\frac{1}{2}}C_s)},\quad \tilde\sigma_k\equiv \bar\epsilon,\quad\forall k\geq1.
\end{aligned}\ee
The following theorem characterizes the oracle complexity of Algorithm \ref{alg1} to reach a stochastic $\epsilon$-KKT point of (\ref{p}).
\begin{theorem}
    Suppose that Assumptions \ref{a1}-\ref{a4}  hold. Let the parameters of Algorithm \ref{alg1} be chosen as in Corollary \ref{lem:estimator-error} along with (\ref{varepsilon-rm}), and set $\tau=\lceil\epsilon^{-1}\rceil$ and
    \begin{align}\label{K-mb}
    K = \left\lceil64\left(\frac{f(x_1)-f_{\rm low}+2MGL_c/\nu^2+\rho_{\max}M^2}{ \eta_{\rm min}}\right)\epsilon^{-2}\right\rceil.
    \end{align}
    Then the algorithm reaches a stochastic $\epsilon$-KKT point within $K$ iterations. Moreover, the expected number of evaluations required for the stochastic objective gradients, stochastic constraint gradients, and constraint values are each of order $\mathcal{O}(\epsilon^{-3})$.
\end{theorem}
\begin{proof}
    For notational simplicity, we define $\mathcal{E}_k:=(\E[\|{\rm P}_{V_k}\nabla f_k\|^2]\big)^{\frac{1}{2}}+(\E[\|c_k\|^2]\big)^{\frac{1}{2}}$. It follows from Jensen's inequality that
    \begin{align*}
         \Big( \frac{1}{K}\sum_{k=1}^{K}\mathcal{E}_k \Big)^2 \le \frac{1}{K}\sum_{k=1}^{K}\mathcal{E}_k^2\le \frac{2}{K}\sum_{k=1}^{K}
    \left(\E[\|{\rm P}_{V_k} \nabla f_k\|^2+\|c_k\|^2]\right),
    \end{align*}
    %\begin{align*}
    %&\bigg( \frac{\sum_{k=1}^{K} (\E[\|{\rm P}_{V_k} \nabla f_k\|^2])^{1/2} + (\E[\|c_k\|^2])^{1/2}}{K} \bigg)^2 \\
    %&\le \frac{1}{K}\sum_{k=1}^{K}
    %\left((\E[\|{\rm P}_{V_k} \nabla f_k\|^2])^{1/2}+(\E[\|c_k\|^2])^{1/2}\right)^2\le \frac{2}{K}\sum_{k=1}^{K}
    %\left(\E[\|{\rm P}_{V_k} \nabla f_k\|^2+\|c_k\|^2]\right),
   %\end{align*}
    which together with   \eqref{eq:s2} gives 
    \begin{align}
        %&\left( \frac{\sum_{k=1}^{K} (\E[\|{\rm P}_{V_k} \nabla f_k\|^2])^{1/2} + (\E[\|c_k\|^2])^{1/2}}{K} \right)^2\label{conv2}\\
        \Big( \frac{1}{K}\sum_{k=1}^{K}\mathcal{E}_k \Big)^2&\le \frac{8\E[\calL(x_1,\rho_1)-\calL(x_{K+1},\rho_{K+1})]+4M^2\rho_{\max}}{\eta_{\min}K}\notag\\
        &\quad+\frac{8\bar\eta C_s^2}{K}\sum_{k=1}^{K}(\eta_{\max}L_{\max}+\frac12)\tilde\sigma_k^2 + 8\bar\eta\tilde C\max\{1,\rho_{\max} L_c\}\frac{1}{K} \sum_{k=1}^{K} \tilde\sigma_k\mathcal{E}_k.\label{conv2} 
    \end{align}
    By the definition of $\mathcal{L}(x,\rho)$ given in \eqref{ALM}, it holds that
    \begin{align*}
        &\E[\calL(x_1,\rho_1)-\calL(x_{K+1},\rho_{K+1})]\\&\leq\E[ f(x_1)+\lambda(x_1)^\top c(x_1)+\frac{\rho_1}{2}\|c(x_1)\|^2-f_{\rm low}-\lambda(x_{K+1})^\top c(x_{K+1})]\\
        &\leq f(x_1)-f_{\rm low}+\frac{2}{\nu^2}MGL_c +\frac{1}{2}\rho_{\max}M^2.
    \end{align*}
     It then follows by solving the quadratic inequality obtained from (\ref{conv2}) that
    \begin{align*}
        %&\frac{1}{K}\sum_{k=1}^{K}\left(\big(\E[\|{\rm P}_{V_k}\nabla f_k\|^2]\big)^{1/2}+\big(\E[\|c_k\|^2]\big)^{1/2}\right)\\
        & \frac{1}{K}\sum_{k=1}^{K}\mathcal{E}_k\\
        &\leq4\bigg(2C_{m}\bar\epsilon+\bigg(\frac{\eta_{\max}L_{\max}+1}{2}\bar\eta C_s^2\bar\epsilon^2+\frac{M^2\rho_{\max}}{4\eta_{\min}K}+\frac{\calL(x_1,\rho_1)-\calL(x_{K+1},\rho_{K+1})}{2\eta_{\min}K}\bigg)^{\frac{1}{2}}\bigg)\\     &\leq4\bigg(2C_{m}\bar\epsilon+ (\eta_{\max} L_{\max}+1)^{\frac{1}{2}}{\bar\eta}^\frac{1}{2}C_s\bar\epsilon+\frac{ (f(x_1)-f_{\rm low}+\frac{2MGL_c}{\nu^2}+\rho_{\max}M^2)^{\frac{1}{2}}}{\eta_{\min}^{1/2} K^{1/2}}\bigg),
    \end{align*}
     where both inequalities are a direct consequence of $(a+b)^{\frac12}\leq a^{\frac12}+b^{\frac12}$. The settings of $\bar\epsilon$ and $K$ in (\ref{varepsilon-rm}) and (\ref{K-mb}) further ensure that
     \begin{align*}
     %\frac{1}{K}\sum_{k=1}^{K}\E\big[\|{\rm P}_{V_k} \nabla f_k\| + \|c_k\|\big]\le
     \mathbb E\big[\|{\rm P}_{V_R}\nabla f_R\|+\|c_R\|\big]\leq\big(\E[\|{\rm P}_{V_R}\nabla f_R\|^2]\big)^{\frac12}
      +\big(\E[\|c_R\|^2]\big)^{\frac12}= \frac{1}{K}\sum_{k=1}^{K}\mathcal{E}_k\leq\epsilon.
    \end{align*}
     Consequently, the algorithm reaches a stochastic $\epsilon$-KKT point in $K$ iterations by (\ref{def-KKT}). Accordingly, the total number of evaluations for the stochastic objective gradient is
    \begin{align}\label{eq:N_f}
        &N_f(K):=\sum_{k=1}^{K} B_k^f \le \left\lceil \frac{K}{\tau}\right\rceil \frac{2\sigma_f^2}{\bar\epsilon^2}+ K+ \frac{2\tau L_f^2}{\bar\epsilon^2}\sum_{k=1}^{K-1}\|x_{k+1}-x_k\|^2.
    \end{align}
    Recall from (\ref{eq:sum-sk}) that
    \begin{align}\label{eq:xk+1-xk}
        &\sum_{k=1}^{K}\E[\|x_{k+1}-x_k\|^2]=\sum_{k=1}^{K}\E[\eta_k^2\|\tilde{s}_k\|^2]\\
        &\leq 8\bar\eta\eta_{\max}\E[\calL(x_1,\rho_1)-\calL(x_{K+1},\rho_{K+1})]+8\bar\eta\eta_{\max}^2 C_s^2 \sum_{k=1}^{K-1} \tilde\sigma_k^2+4\bar\eta\eta_{\max}M^2\rho_{\max}\notag\\
        &\quad+ 8\bar\eta\eta_{\max}^2\tilde C\max\{1,\rho_{\max} L_c\}\sum_{k=1}^{K}\tilde\sigma_k\mathcal{E}_k\notag\\
        &\leq 8\bar\eta\eta_{\max}\big(f(x_1)-f_{\rm low}+\frac{2MGL_c}{\nu^2}+\rho_{\max}M^2\big)+8K\bar\eta\eta_{\max}^2(C_s^2\bar\epsilon^2+C_m\bar\epsilon\epsilon)=\mathcal{O}(1).\notag
    \end{align}
   Taking the expectation of (\ref{eq:N_f}) in conjunction with \eqref{eq:xk+1-xk} leads to
    \begin{align*}
        \E[N_f(K)]=\mathcal{O}(\epsilon^{-3}).
    \end{align*}
     The same argument applies to stochastic gradient and function value evaluations of the constraints, yielding the stated oracle complexity of the proposed method.
\end{proof}
\begin{remark}
The oracle complexity result is established in expectation, which is consistent with the present analysis framework. The argument based on the good/bad event decomposition requires per-iteration control of the bad-event probability, as shown in Lemma \ref{lem:bad-event}. Hence, the analysis relies on conditional second-moment bounds for the stochastic estimators at each iteration, rather than only an averaged error bound along the trajectory. For the SPIDER-type recursion, such a requirement naturally  leads to batch sizes depending on the local increment $\|x_k-x_{k-1}\|^2$, as reflected in Corollary \ref{lem:estimator-error}. Since these local increments are controlled only after summing over the iterates and taking expectations, the resulting oracle complexity statement is obtained in expectation.
\end{remark}
\section{Numerical simulation}
In this section, we assess the numerical performance of Algorithm \ref{alg1} by comparing it with several existing algorithms for solving \eqref{p} on test problems from the CUTEst collection \cite{GOT2015} and a nonconvex fairness-constrained problem \cite{LCLLX2024,MLY2020}.

\subsection{Experiments on CUTEst problems}
In this experiment, we construct a benchmark set of $138$ equality-constrained problems from CUTEst. All the test problems satisfy $n+m\le 1000$, where $n$ and $m$ denote the numbers of variables and constraints, respectively.  %{Specifically, these instances are selected according to the following criteria: (i) successful initialization} 
To emulate the fully-stochastic setting considered in this paper, we generate stochastic oracle samples from the deterministic CUTEst problems by adding independent zero-mean Gaussian perturbations to the objective gradient, the constraint value, and the constraint Jacobian. We conduct two sets of experiments, with the noise level fixed at \(10^{-2}\) and \(10^{-4}\), respectively.

{We compare the proposed Algorithm \ref{alg1} (shorted as FSFO) with Stoc-iALM \cite{LCLLX2024}, as well as Algorithm 2 in \cite{AW2024} and Algorithm 1 in \cite{SZZ2025}, referred to as SLQPM and SSQP, respectively.} Throughout the experiments, the comparison is made under the same stochastic sample budget for all methods. 
A run is terminated when the total number of stochastic samples
\[N_{\rm sam}:= N_{\nabla f}+N_c+N_{\nabla c}\]
reaches $3000$, where \(N_{\nabla f}\), \(N_c\), and \(N_{\nabla c}\) count the total numbers of samples used to approximate objective gradients, constraint values, and constraint Jacobians, respectively. For each method, the batch size is selected from \(\{1,2,5,10,20\}\), and the same random seed is used for all methods on each test problem. 

To make the parameter tuning comparable across methods, each method is assigned the same tuning budget of $75$ parameter candidates on each test problem. The candidate sets are chosen separately for different methods, but the tuned parameters are restricted to those playing the main roles in the corresponding algorithms. For FSFO, we tune the normal feasibility weight \(w\) and the stepsize scaling parameter \(\eta\). For SSQP, we tune two auxiliary parameters used to determine the stepsize. For SLQPM, we tune the gradient stepsize and the recursive-momentum parameter. For Stoc-iALM, we tune the stepsize scale and the parameters controlling the penalty and multiplier updates, as well as the inner-loop budget. For each method, we choose the best parameter candidate on each test problem according to a feasibility-prioritized criterion. Specifically, we identify the iterate with the smallest stationarity error among all generated iterates whose feasibility error is no larger than \(10^{-4}\), and select the corresponding parameter candidate. 
If this feasibility threshold is not reached, we instead select the parameter candidate corresponding to the iterate with the smallest feasibility error. To assess the quality of the returned points across different methods, we define the score of the selected iterate by the aggregate KKT residual
\[
  \max\{\|\nabla f(x)+\nabla c(x)\lambda\|_\infty,\,\|c(x)\|_\infty\}.
\]
\paragraph{Results}
Table \ref{tab:semantic-fair-summary} summarizes the numerical comparison under noise levels \(10^{-2}\) and \(10^{-4}\), respectively. 
In each table, the column ``wins'' reports the number of test problems on which a method attains the smallest aggregate score among the four methods. 
Figure \ref{fig:semantic-fair-boxplot-2}  reports the distributions of the stationarity and feasibility errors under the two noise levels. 

\begin{table}[htbp]
\caption{Summary of the comparison under two noise levels.}
\label{tab:semantic-fair-summary}
\centering
\normalsize
\setlength{\tabcolsep}{4pt}
\begin{tabular}{cccccc}
\hline
Noise & Algorithm & Wins & Median score & Median stationarity & Median feasibility \\
\hline
\multirow{4}{*}{\(10^{-2}\)}
& FSFO    & 103 & \(1.997\times10^{-3}\) & \(3.331\times10^{-4}\) & \(1.404\times10^{-3}\) \\
& SSQP      & 25  & \(3.454\times10^{-3}\) & \(1.132\times10^{-3}\) & \(2.135\times10^{-3}\) \\
& SLQPM     & 12  & \(8.388\times10^{-2}\) & \(9.467\times10^{-3}\) & \(1.640\times10^{-2}\) \\
& Stoc-iALM & 8   & \(9.364\times10^{-2}\) & \(9.534\times10^{-3}\) & \(4.634\times10^{-2}\) \\
\hline
\multirow{4}{*}{\(10^{-4}\)}
& FSFO    & 104 & \(2.606\times10^{-4}\) & \(5.913\times10^{-5}\) & \(1.191\times10^{-4}\) \\
& SSQP      & 26  & \(5.829\times10^{-4}\) & \(1.413\times10^{-4}\) & \(2.797\times10^{-4}\) \\
& SLQPM     & 10  & \(7.993\times10^{-2}\) & \(9.482\times10^{-3}\) & \(1.397\times10^{-2}\) \\
& Stoc-iALM & 7   & \(9.164\times10^{-2}\) & \(2.654\times10^{-3}\) & \(3.681\times10^{-2}\) \\
\hline
\end{tabular}
\end{table}
\begin{figure}[htbp]
\centering
\includegraphics[width=\textwidth]{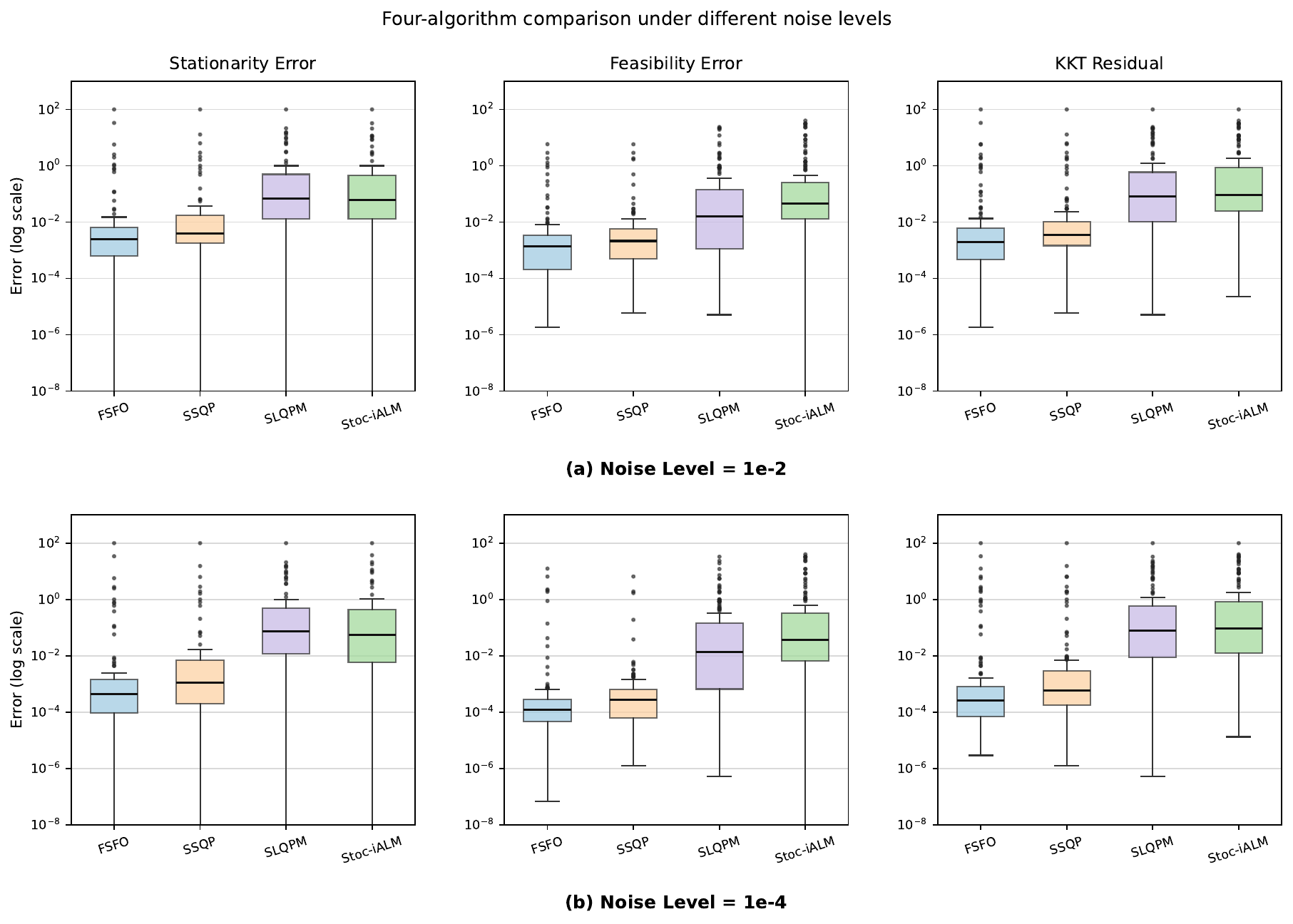}
\vspace{-1em}
\caption{Box plots of stationarity, feasibility, and KKT residuals under different noise levels.}
\label{fig:semantic-fair-boxplot-2}
\end{figure}

These results show that FSFO provides a balanced reduction of the two residuals, with more wins and lower median values for both KKT (stationarity and feasibility) errors under both noise levels, as reported in Table \ref{tab:semantic-fair-summary}. The boxplots in Figure \ref{fig:semantic-fair-boxplot-2}  further show that the three residual distributions of FSFO are concentrated at relatively low levels over the selected problems. SSQP shows distributional behavior comparable to FSFO, although the median score and the two median residuals reported in Table \ref{tab:semantic-fair-summary} are slightly larger. SLQPM and Stoc-iALM are less competitive in these experiments, showing relatively larger residuals across the three error measures. Overall, the results indicate that FSFO exhibits favorable numerical performance in these experiments.

\subsection{Nonconvex fairness constrained problem}
We next examine the performance of Algorithm \ref{alg1} on a fairness-constrained classification problem with a nonlinear constraint studied in \cite{LCLLX2024,MLY2020}. Let $\mathcal{D}:=\{(a_i,b_i)\}_{i=1}^{|\mathcal{D}|}$ be a labeled dataset, where $a_i\in\mathbb{R}^n$ and $b_i\in\{-1,1\}$ denote the feature vector and the corresponding class label, respectively. 
For $x\in\mathbb{R}^n$, define 
$\ell(x;a,b):=\log(1+\exp(-ba^\top x))$ and 
$F(x;a,b):=\varphi_\alpha(\ell(x;a,b))$, where 
$\varphi_\alpha(t)=\alpha\log(1+t/\alpha)$ with $\alpha=2$. Let $\mathcal{S}=\{a_j\}_{j=1}^{N_S}$ be a possibly unlabeled dataset, with $\mathcal{S}_{\min}\subseteq\mathcal{S}$ denoting the subset corresponding to the minority group. Following \cite{LCLLX2024}, the resulting nonconvex fairness-constrained classification model is given by
\begin{equation}
\label{eq:fairness}
\begin{aligned}
    \min_{x\in\mathbb{R}^n}\quad 
    & f(x):=\frac{1}{|\mathcal{D}|}\sum_{(a,b)\in\mathcal{D}} F(x;a,b),\\
    \mathrm{s.t.}\quad
    & c(x):=\omega\sum_{a\in\mathcal{S}}\sigma(a^\top x)
    -
    \sum_{a\in\mathcal{S}_{\min}}\sigma(a^\top x)
    \leq 0,
\end{aligned}
\end{equation}
where $\sigma(t):=\exp(t)/(1+\exp(t))$ and $\omega\in(0,1)$ is a prescribed fairness parameter. The fairness constraint in (\ref{eq:fairness}) is imposed to ensure that the classifier yields a sufficient level of positive predictions for the minority group. We test this problem on the $a9a$ dataset from the LIBSVM library \cite{LIBSVM}, with $n=123$ and $(|\mathcal{D}|,|\mathcal{S}|,|\mathcal{S}_{\min}|)=(32561,16281,1561)$, and on the $loan$ dataset from LendingClub used in \cite{LCLLX2024}, with $n=250$ and $(|\mathcal{D}|,|\mathcal{S}|,|\mathcal{S}_{\min}|)=(64485,63890,31966)$. We set $\omega=0.3$ for $a9a$ and $\omega=0.5$ for $loan$, and compare Algorithm \ref{alg1} (FSFO) with Stoc-iALM \cite{LCLLX2024}, SSQP \cite{SZZ2025}, and SLQPM \cite{AW2024}.

In the numerical experiments, following \cite{LCLLX2024}, we introduce a slack variable $s\geq 0$ and reformulate the inequality constraint as $c(x)+s=0$ for the algorithmic updates. The total sample budget is set to \(4\times 10^4\) on $a9a$ and \(2\times 10^4\) on $loan$ for all methods. We summarize below the key parameter settings used in the experiments. For FSFO, we set the refresh period of the SPIDER estimators to $\tau=2$ on both datasets, use the same batch sizes in the refresh and recursive steps, and set the stepsize as $\eta_k=0.25c_\eta/L_k$. On $a9a$, we set $\rho_0=0.5$, $c_\eta=175$ and use $(B^f_k,B^c_k,B^J_k)=(12,64,64)$, while on $loan$, we set $\rho_0=0.26$, $c_\eta=320$ and $(B^f_k,B^c_k,B^J_k)=(12,18,18)$. For SSQP, the objective gradient, constraint value, and constraint Jacobian are estimated using mini-batches. On $a9a$, we use $(B^f_k,B^c_k,B^J_k)=(16,64,64)$ with initial stepsize parameter $\beta=0.9$, initial merit parameter $\tau_0=0.08$, and Hessian approximation $H=30I$. On $loan$, we use $(B^f_k,B^c_k,B^J_k)=(16,16,16)$ and set $\beta=1$, $\tau_0=0.1$, $H=1.5I$. For SLQPM, on the $a9a$ dataset, the penalty parameter, stepsize, and momentum parameter are given by $\rho_k=\rho k^{0.4}$, $\eta_k=\eta/(\rho_k( k+1)^{0.35})$, $\alpha_k=\alpha k^{-0.8}$, where $\rho=8$, $\eta=0.75$ and $\alpha=72/81$. On the $loan$ dataset, we set $\rho_k=\rho k^{0.2}$, $\eta_k=\eta/(\rho_k( k+1)^{0.6})$ with $\rho=1.2$, $\eta=0.25$, while keeping the setting of $\alpha_k$ unchanged. For Stoc-iALM, we thank the authors of \cite{LCLLX2024} for kindly sharing their implementation, from which we retain several parameter settings while making the adjustments required for our experimental setup. On $loan$, we retain the penalty update with $\beta_0=1$ and growth factor $2.5$, and use a PStorm mini-batch size of $30$. On $a9a$, we set $\beta_0=120$, the growth factor to $1.55$, and the PStorm inner batch size to $20$. We record the objective value $f(x)$, the original signed constraint value $c(x)$, and the KKT residual for \eqref{eq:fairness} as the cumulative number of stochastic oracle samples increases. To evaluate the KKT residual, for each iterate $x$, we compute a nonnegative multiplier by solving ${\min}_{z\geq0}\, \left\{ \|\nabla f(x)+z\nabla c(x)\|^2 + |zc(x)|^2 \right\},$
and define 
\begin{align*}
    \operatorname{Res}_{\rm KKT}(x) := \left( \|\nabla f(x)+z\nabla c(x)\|^2 + [c(x)]_+^2 + |zc(x)|^2 \right)^{1/2},
\end{align*}
where $[t]_+:=\max\{t,0\}$.  As the algorithms are randomized, the numerical results reported below are averaged over five independent runs with different random seeds for each method, with the shaded regions representing the corresponding standard deviations.

\begin{figure}[htbp]
\centering
\includegraphics[width=\textwidth]{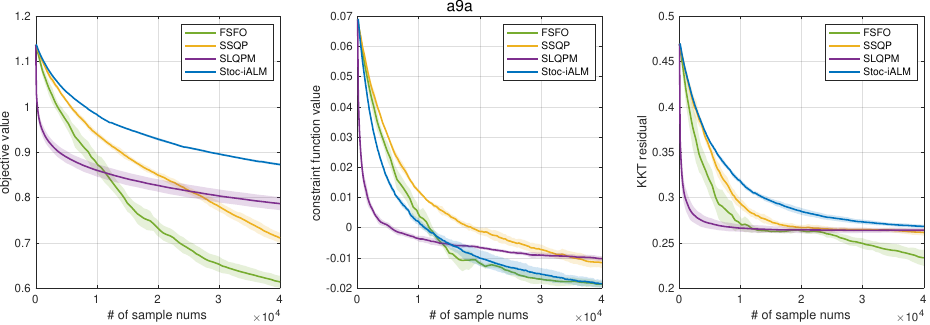}
\vspace{-1em}
\caption{Comparison of  FSFO, SLQPM, SSQP and Stoc-iALM on the $a9a$ dataset.}
\label{fig:a9a}
\end{figure}
%\vspace{-1em}
\begin{figure}[htbp]
\centering
\includegraphics[width=\textwidth]{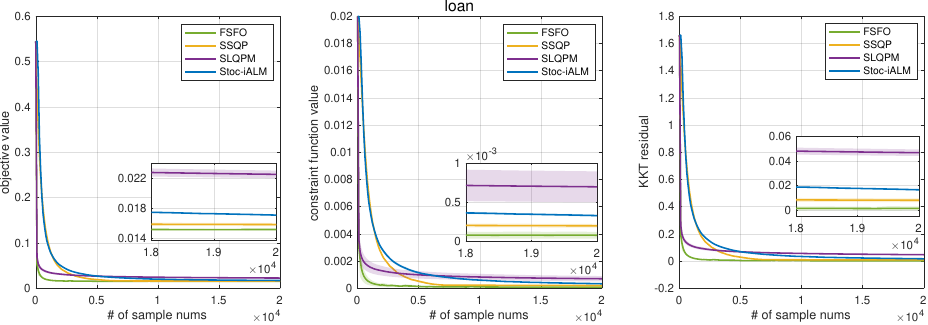}
\vspace{-1em}
\caption{Comparison of  FSFO, SLQPM, SSQP and Stoc-iALM on the $loan$ dataset.}
\label{fig:loan}
\end{figure}
Figures \ref{fig:a9a} and \ref{fig:loan} present the performance of FSFO in comparison with the three other algorithms on the $a9a$ and $loan$ datasets, respectively. As shown in Figure \ref{fig:a9a}, while all methods eventually satisfy the inequality constraint on the $a9a$ dataset, FSFO yields the lowest objective value and KKT residual within the prescribed sample budget. For the $loan$ dataset, the results in Figure \ref{fig:loan} indicate that all methods rapidly reduce the three measures during the initial stage and exhibit similar performance overall, while FSFO achieves a somewhat faster reduction and slightly lower terminal values, as further illustrated by the enlarged views.
In summary, these results suggest that FSFO is competitive with the existing methods in the fairness-constrained classification experiment.

\section{Conclusion}
This paper studies nonconvex stochastic equality-constrained optimization and proposes {a stochastic first-order method.} At each iteration, the method computes a search direction consisting of a tangential descent component and a normal component {for feasibility improvement using truncated SPIDER-type recursive estimators}, and employ the  Fletcher's augmented Lagrangian as a smooth merit function to updates the step size. 
{Under the additional Lipschitz continuity of the second order derivatives of the objective and constraint functions, together with a smallest singular value-based event decomposition, we derive that} the method attains a stochastic $\epsilon$-KKT point within $\mathcal{O}(\epsilon^{-2})$ iterations with an expected oracle complexity of $\mathcal{O}(\epsilon^{-3})$. This matches the best known stochastic gradient complexity and improves the known order for stochastic constraint value evaluations in the fully-stochastic setting. We finally report numerical results to illustrate the performance of the proposed method.

\appendix
\section{Auxiliary Lemmas}\label{app:lemma}

\begin{lemma}[Sherman-Morrison-Woodbury Formula]\label{lm:smw}
Let \( A, B \in\R^{m\times m} \) be invertible matrices, then
\(
A^{-1} - B^{-1} = A^{-1} (B - A) B^{-1}.
\)
\end{lemma}

\begin{lemma}[Weyl's inequality]
\label{lem:weyl}
For any \(A,B \in \mathbb{R}^{m\times n}\), it holds that
\(|\mu_i(A)-\mu_i(B)| \le \|A-B\|_2,\) \( i=1,\ldots,\min\{m,n\},\) and \(
\mu_{\min}(A)\ge \mu_{\min}(B)-\|A-B\|_2,\) 
where \(\mu_i(\cdot)\) denotes the \(i\)-th singular value and \(\mu_{\min}\) refers to the smallest singular value.
\end{lemma}

\section{Supported Proofs}\label{app:proof}
\begin{proof}[Proof of Proposition \ref{le:lambda}]
   For simplicity, we denote $A(x):=(\nabla c(x)^\top \nabla c(x))^{-1}$,  $b(x):=\nabla c(x)^\top\nabla f(x)$ and $B(x):=A(x)^{-1}$. It follows from Lemma \ref{lm:smw} that
   \begin{align}\label{Ax-Ay}
       \|A(x)-A(y)\|\leq\|A(x)\| \|\nabla c(x)^\top \nabla c(x) - \nabla c(y)^\top \nabla c(y)\| \|A(y)\|.
   \end{align}
   Since $\lambda(x)=-A(x)b(x)$ is well defined under Assumption \ref{a4}, it holds that
\begin{align*}
&\|\lambda(x)-\lambda(y)\|=
\|A(x)b(x) - A(y)b(y)\| \\
&\leq
\|A(x)\nabla c(x)^\top - A(y)\nabla c(y)^\top\| \|\nabla f(y)\| + \|A(x)\nabla c(x)^\top (\nabla f(x) - \nabla f(y))\| \\
&\leq
\|A(x) - A(y)\| \|\nabla c(x)\| \|\nabla f(y)\| + \|A(y)\| \|\nabla c(x) - \nabla c(y)\| \|\nabla f(y)\| \\
&\quad + \|A(x)\| \|\nabla c(x)\| \|\nabla f(x) - \nabla f(y)\| \\
&\leq
\|A(x)\| \|\nabla c(x)^\top \nabla c(x) - \nabla c(y)^\top \nabla c(y)\| \|A(y)\| \|\nabla c(x)\| \|\nabla f(y)\| \\
&\quad + \|A(y)\| \|\nabla c(x) - \nabla c(y)\| \|\nabla f(y)\| + \|A(x)\| \|\nabla c(x)\| \|\nabla f(x) - \nabla f(y)\| \\
&\leq
\frac{2GL_c^2 L_J}{\nu^4}\|x-y\| + \frac{G L_J}{\nu^2}\|x-y\| + \frac{L_c L_f}{\nu^2}\|x-y\|:=L_\lambda\|x-y\|,\quad \forall x,y,
\end{align*}
where the third inequality follows from (\ref{Ax-Ay}), thereby establishing the first part of the result. Attention is now turned to the second part. It follows from Assumptions \ref{a2}-\ref{a3} that 
\begin{align*}
    \|\nabla B(x) - \nabla B(y)\|&\leq2\|\nabla^2 c(x)\|\|\nabla c(x)-\nabla c(y)\|+2\|\nabla c(y)\|\|\nabla^2 c(x)-\nabla^2 c(y)\|\\&\leq2(L_J^2+L_cL_h^c)\|x-y\|,
\end{align*}
which in turn implies 
\begin{align*}
    &\|\nabla A(x)-\nabla A(y)\|=\|A(x)\nabla B(x)A(x)-A(y)\nabla B(y)A(y)\|\\
    &\leq\|A(x)-A(y)\|\|\nabla B(x)\|\|A(x)\|+\|A(y)\|\|\nabla B(x) - \nabla B(y)\|A(x)\|\\
    &\quad+\|A(y)\|\|\nabla B(y)\|\|A(x) - A(y)\|\\
    &\leq\frac{4L_JL_c}{\nu^2}\|A(x) - A(y)\|+\frac{1}{\nu^4}\|\nabla B(x) - \nabla B(y)\|\\
    &\leq \frac{8L_J^2L_c^2}{\nu^6}\|x-y\|+\frac{2(L_J^2+L_cL_h^c)}{\nu^4}\|x-y\|.
\end{align*}
In addition, one has
\begin{align*}
    &\|\nabla b(x)-\nabla b(y)\|\\
    &=\|\nabla^2c(x)\nabla f(x)+\nabla c(x)^\top\nabla^2f(x)-\nabla^2c(y)\nabla f(y)-\nabla c(y)^\top\nabla^2f(y)\|\\
    &\leq \|\nabla^2c(x)\|\|\nabla f(x) - \nabla f(y)\| + \|\nabla^2c(x) - \nabla^2c(y)\|\|\nabla f(y)\| \\ 
    &\quad + \|\nabla c(x)\|\|\nabla^2f(x) - \nabla^2f(y)\| + \|\nabla c(x) - \nabla c(y)\|\|\nabla^2f(y)\| \\ 
    &\leq L_JL_f\|x-y\| + L_h^cG\|x-y\|+ L_cL_h^f\|x-y\| + L_JL_f\|x-y\| \\ 
    &= (2L_JL_f + L_h^cG + L_cL_h^f)\|x-y\|.
\end{align*}
Hence, together with
\begin{align*}
    \|b(x)-b(y)\|&\leq\|\nabla c(x)\|\|\nabla f(x)-\nabla f(y)\|+\|\nabla c(x)-\nabla c(y)\|\|\nabla f(x)\|\\
    &\leq (L_cL_f+GL_J)\|x-y\|,
\end{align*}
we can derive
\begin{align*}
    &\|\nabla\lambda(x)-\nabla\lambda(y)\|\\
    &=\|\nabla A(x)b(x)+A(x)\nabla b(x)-\nabla A(y)b(y)-A(y)\nabla b(y)\|\\
    &\leq\|\nabla A(x)\|\|b(x)-b(y)\|+\|\nabla A(x)-\nabla A(y)\|\|b(y)\|+\|A(x)\| \|\nabla b(x) - \nabla b(y)\| \\
    &\quad+ \|A(x) - A(y)\| \|\nabla b(y)\|\\
    &\leq\frac{2L_JL_c(L_cL_f+GL_J)}{\nu^4}\|x-y\|+GL_c\left(\frac{8L_J^2L_c^2}{\nu^6}+\frac{2(L_J^2+L_cL_h^c)}{\nu^4}\right)\|x-y\|\\
    &\quad+\frac{2L_JL_f + L_h^cG+ L_cL_h^f}{\nu^2}\|x-y\|+\frac{2L_JL_c(GL_J+ L_cL_f)}{\nu^4}\|x-y\|:=L_\lambda^1\|x-y\|,
\end{align*}
which yields the conclusion. 

To prove the Lipschitz continuity of $\nabla \mathcal L(x,\rho)$ in \(x\), for notational convenience we define 
$\psi(x,\rho):=\langle\lambda(x), c(x)\rangle+{\frac{\rho}{2}}\|c(x)\|^{2}$. 
It then follows that
    \begin{align*}
        \|\nabla\psi(y,\rho)-\nabla\psi(x,\rho)\|
&=\|\nabla\lambda(y)c(y)+\nabla c(y)\lambda(y)+\rho\nabla c(y)c(y)-\nabla\lambda(x)c(x)\\
        &\qquad-\nabla c(x)\lambda(x)-\rho\nabla c(x)c(x)\|\\
        &\leq\|\nabla\lambda(y)\|\|c(y)-c(x_k)\|+\|\nabla\lambda(y)-\nabla\lambda(x)\|\|c(x)\|\\
        &\quad+\|\lambda(y)-\lambda(x)\|\|\nabla c(y)\|+\|\lambda(x)\|\|\nabla c(y)-\nabla c(x)\|\\
        &\quad+\rho\|\nabla c(y)\|\|c(y)-c(x)\|+\rho\|\nabla c(y)-\nabla c(x)\|\|c(x)\|\\
        &\leq \Big(2L_\lambda L_c+ML_\lambda^1+\frac{GL_cL_J}{\nu^2}+\rho (L_c^2+ML_J)\Big)\|y-x\|,
    \end{align*}
    which further indicates
    \begin{align*}
        \|\nabla\calL(y,\rho)-\nabla\calL(x,\rho)\|
        &\leq\|\nabla f(y)-\nabla f(x)\|+\|\nabla\psi(y,\rho)-\nabla\psi(x,\rho)\|\\
        &\leq\Big(L_f+2L_\lambda L_c+ML_\lambda^1+\frac{GL_cL_J}{\nu^2}+\rho (L_c^2+ML_J)\Big)\|y-x\|,
    \end{align*}
    which completes the proof.
\end{proof}

\bibliographystyle{abbrv}
\bibliography{paper_assets/references}

\end{document}